\documentclass[twoside,11pt]{amsart}

\usepackage{amsmath,latexsym,amssymb,mathptm, times,verbatim, enumerate}
    
    \usepackage{tikz-cd}
\usepackage{xcolor, amsthm}
\usepackage{tikz}
\usepackage{multicol}
\usetikzlibrary{matrix}
\usepackage{wrapfig}
\usepackage{wasysym}
\usepackage{cleveref}

\input amssym.def
\input amssym
\input xypic
\input xy
\xyoption{all}
\def\Sym{\operatorname{Sym}}

\def\L{\mathcal L}
\def\J{\mathcal J}

\def\m{\mathfrak m}
\def\R{\mathcal R}

\def\F{\mathcal F}

\def\min{{\mathrm {min}}}
\def\max{{\mathrm {max}}}
\def\det{{\mathrm {det}}}

\def\init{{\mathrm {in}}}

\def\Spec{{\mathrm {Spec}}}

\def\Tor{{\mathrm {Tor}}}

\newtheorem{theorem}{Theorem}[section]
\newtheorem{lemma}[theorem]{Lemma}
\newtheorem{corollary}[theorem]{Corollary}

\newtheorem{criteria}[theorem]{Criteria}

\newtheorem{prop}[theorem]{Proposition}
\newtheorem{example}[theorem]{Example}
\newtheorem*{theorem*}{Theorem}
\newtheorem*{mthm*}{Main Theorem}

\newtheorem*{rep@theorem}{\rep@title}
\newcommand{\newreptheorem}[2]{%
\newenvironment{rep#1}[1]{%
 \def\rep@title{#2 \ref{##1}}%
 \begin{rep@theorem}}%
 {\end{rep@theorem}}}
\makeatother

\newreptheorem{theorem}{Theorem}

\newtheorem{def&dis}[theorem]{Definition and Discussion}
\newtheorem{definition}[theorem]{Definition}
\newtheorem{remark}[theorem]{Remark}

\newtheorem{data}[theorem]{Data}

\numberwithin{equation}{theorem}

\begin{document}

\title{Rees algebras of ideals submaximally generated by quadrics}

\author{Whitney Liske}
\address{Department of Mathematics, Saint Vincent College 
300 Fraser Purchase Rd, Latrobe, PA 15650}
\email{whitney.liske@stvincent.edu}
\thanks{The author was partially supported by Mathematical Endeavors Revitalization Program of the Association for Women in Mathematics}

\begin{abstract}
The goal of this paper is to study the Rees algebra $\R(I)$and the special fiber ring $\F(I)$ for a family of ideals. 
Let $R=\mathbb{K}[x_1, \ldots, x_d]$ with $d\geq 4$ be a polynomial ring with homogeneous maximal ideal $\m$. 
We study the $R$-ideals $I$, which are $\m$-primary, Gorenstein, generated in degree 2, and have a Gorenstein linear
resolution. 
In the smallest case, $d=4$, this family includes the ideals of $2\times 2$ minors of a general $3\times 3$ matrix of linear forms in $R$.
We show that the defining ideal of the Rees algebra will be of fiber type. 
That is, the defining ideal of the Rees algebra is generated by the defining ideals of the special fiber ring and of the symmetric algebra. 
We use the fact that these ideals differ from $\m^2$ by exactly one minimal generator to describe the defining ideal $\mathcal{F}(I)$ as a sub-ideal of the defining ideal of $\mathcal{F}(\m^2)$, which is well known to be the ideal of $2\times 2$ minors of a symmetric matrix of variables.
\end{abstract}
\maketitle

\section{Introduction}

In this paper, we will study the Rees algebra $\R(I)=R[It]$ and the special fiber ring $\mathcal{F}(I)=\R(I)\otimes \mathbb{K}$ of an ideal $I$ in a polynomial ring $R$ over a field $\mathbb{K}$.
Rees algebras are an important tool in commutative algebra, algebraic geometry, elimination theory, intersection theory, and geometric modeling. 
They hold asymptotic information about the powers of an ideal.
Rees algebras also provide a natural way to look at the multiplicity of a ring. 
They are used in algebraic geometry because they define the blowup of the scheme $\Spec(R)$ along the closed subscheme $V(I)$.
In addition, the special fiber ring gives the coordinate ring of algebraic varieties, such as Segre, Veronese, Grassmannian, Pfaffian, and determinantal varieties. 
These classical algebraic varieties are of fundamental importance as they are ubiquitous and play an important role in almost all fields of mathematics as well as in theoretical computer science, complexity theory, signal processing, phylogenetics, and algebraic statistics.
In particular, the ideals studied in this paper correspond to projections of degree two Veronese varieties.

One of the main problems in this area is to describe $\R(I)$ and $\mathcal{F}(I)$ in terms of generators and relations.
That is, find polynomial rings $S=R[w_1, \ldots, w_n]$ and $W=\mathbb{K}[w_1, \ldots, w_n]$ and ideals $\J$ and $I(X)$ such that $\R(I)=S/\J$ and $\mathcal{F}(I)=W/I(X)$.
We call $\J$ and $I(X)$ the defining ideals of $\R(I)$ and $\mathcal{F}(I)$, respectively.
If the ideal $I$ is generated in the same degree, then the generators define a rational map between projective spaces.
The Rees algebra and special fiber ring will then be the coordinate rings of the graph and image of this map, respectively.
By computing the defining ideals we find the implicit equations of the variety defined by the graph and image.
This particular problem is also studied by applied mathematicians, where it is known as the implicitization problem (for example, \cite{BJ}, \cite{C}, \cite{EM}).

Indeed, this is a very challenging problem and is still open for many classes of ideals that seem innocent at first glance.
Significant restrictions typically need to be placed on the ideal in order to compute the defining ideal of the Rees algebra.
When the ideal is defined by maximal minors of a generic matrix, the special fiber ring is classically known to be the coordinate ring of a Grassmannian variety, and is known to be defined by the Plucker relations.
When the minors are nonmaximal, much less is known.
In  \cite{BCV}, Bruns, Conca, and Varbaro use representation theory on the general linear group to compute some degree two and three relations on the $t$-minors of a generic matrix and conjectured these defined the fiber ring for certain matrix sizes. 
This conjecture was later proven for $t=2$ in \cite{HPPRS} by Huang, Perlman, Polini, Raicu and Sammartano.

There has been a significant body of work studying the case of height two perfect ideals.
The defining ideals are given by Morey and Ulrich in this case when $I$ is linearly presented and satisfies $G_d$ (a condition on the height of fitting ideals) (\cite{MU}).
These results were later generalized, typically by studying either ideals that are not linearly presented ( \cite{KPU.R}, \cite{L}) or do not satisfy $G_d$ ( \cite{M}, \cite{H}), but there are still open cases.
Another important generalization of these results, studied by Kustin, Polini, and Ulrich in \cite{KPU.T}, are linearly presented height three Gorenstein ideals satisfying $G_d$.
The proof of this result requires local cohomology techniques, which they developed in \cite{KPU.D}.


This paper will focus on certain Gorenstein ideals generated by quadrics studied in a 2015 paper by Hong, Simis, and Vasconcelos \cite{HSV}. 
Let $R=\mathbb{K}[x_1, \ldots, x_d]$ be a polynomial ring in $d\geq 4$ variables over a field $\mathbb{K}$. 
Let $\m$ be the homogenous maximal ideal of $R$. 
An ideal $I$ of $R$ generated by quadrics is said to be {\em submaximally generated} if it has codimension $d$ and is minimally generated by $n=\nu(\m^2)-1 $ quadrics.
Theorem 3.5 in \cite{HSV} shows that, after a potential change of variables,  $I =<x_ix_j , x_k^2-x_d^2| i\neq j , k\neq d>$. 
Further, these ideals are $\m$-primary, Gorenstein, generated in degree 2, and have a Gorenstein linear resolution. 
In the smallest case, $d=4$, this family includes the ideals of $2\times 2$ minors of a general $3\times 3$ matrix of linear forms in $R$.
Our goal is to compute the defining ideals of the special fiber ring and Rees algebra for these ideals.

Ideals submaximally generated by quadrics $I$ have a Gorenstein linear resolution, which can be verified by computing the socle degree of $R/I$ in \cite{HSV}.
Therefore, by a result of Eisenbud, Huneke, and Ulrich \cite{EHU}, it follows that the Rees algebra is of fiber type. 
This means that the defining ideal of the Rees algebra is determined by the defining ideals of the symmetric algebra and special fiber ring. 
Since the defining ideal of the symmetric algebra can be computed from a minimal homogeneous presentation matrix of the forms $g_{ij}=x_ix_j$ if $i\neq j$ and $g_{ii}=x_i^2-x_d^2$ which generate $I$, we need only compute the defining ideal of the special fiber ring. 

\Cref{background} will formalize the definitions and notations used throughout the paper. 
We construct a candidate $\Lambda$ for the defining ideal of the special fiber ring $\F(I)$ in \Cref{candidate}.
This ideal is designed to mimic the defining ideal of $\F(\m^2)$ as much as possible, in particular, its generators will be linear combinations of $2\times 2 $ minors of a symmetric matrix of variables. 
We also show in \Cref{candidate} that we have $\Lambda\subseteq I(X)$.

To show that these ideals are, in fact, equal we compute invariants for $I(X)$, the defining ideal of $\mathcal F(I)$, by exploiting the connection to $\mathcal F(\m^2)$ in  \Cref{sec:fiber}.
We study $\Tor_1^W(I(X),\mathbb{K})$ in this Section to conclude that $I(X)$ is generated in degrees two and three. 
The main theorem of the paper will then follow once we have shown the equality of $I(X)$ and $\Lambda$ in degrees two and three.
We do this by showing the Hilbert function of $I(X)$ and $\Lambda$ agree in those degrees.
Recall that if $\mathbb{K}$ is a field and $A=\displaystyle \oplus_{i\geq 0} A_i$ is a graded $\mathbb{K}$-algebra, the \emph{Hilbert function} of $A$ in degree $i$ is $HF_A(i)=\dim_\mathbb{K} A_i$. 
While the Hilbert function of $I(X)$ is computed in  \Cref{sec:fiber}, the computation for the candidate $\Lambda$ is more involved.
However, although it is laborious, it is mostly straightforward.
The degree two calculation is in  \Cref{hf2} and the degree three calculation is in  \Cref{hf3}.
In particular, the proving strategy will presented in this paper shows two $\mathbb{K}$ vector spaces given by $I(X), \Lambda$ , one contained in the other, which will agree in the generating degrees of the larger ideals. 
This implies the equality.

The main result is summarized below.

\begin{mthm*}[\Cref{maintheorem1}]
Let $R=\mathbb{K}[x_1, \ldots, x_d]$ be a polynomial ring in $d\geq 4$ variables over a field $\mathbb{K}$ of characteristic zero. 
Let $I =<x_ix_j , x_k^2-x_d^2 \text{ s.t. } i\neq j , k\neq d>$ be a Gorenstein $R$-ideal submaximally generated by quadrics.
Let $\mathcal M=[w_{ij}]$ be a $d\times d$ symmetric matrix with $w_{ij}=w_{ji}$ and $w_{dd}=0$.
We describe the generators of an ideal $\Lambda \subsetneq I_2(\mathcal M)$.
We then show $\Lambda$ is the defining ideal of $\mathcal F(I)$, that is $\Lambda=I(X)$.
 Further $\J=\L+ \Lambda S$ is the defining ideal of $\R(I)$. 
 \end{mthm*}

\section{Background}\label{background}
The purpose of this section is to formalize the definitions and notation used throughout the paper. 
 \Cref{familydata} will set the notation used for the family of ideals studied in this paper and will be consistently adopted in the paper.
We also reference theorems about the blowup algebras and the defining ideal of $\F(\m^2)$, which is well understood.

\begin{data}\label{familydata}
Let $R=\mathbb{K}[x_1, \ldots, x_d]$ be a polynomial ring in $d\geq 4$ variables over a field $\mathbb{K}$ of characteristic zero. 
Let $\m =(x_1, \ldots, x_d)$ be the homogenous maximal ideal of $R$. 
Let \[I =<x_ix_j , x_k^2-x_d^2 \text{ s.t. } i\neq j , k\neq d>\] be a Gorenstein $R$-ideal submaximally generated by quadrics.
Let $W,S$ be the polynomial rings defined below
\[W=\mathbb{K}[w_{ij} ] /<w_{ij}-w_{ji} , w_{dd} \text{ s.t. } i\neq j \in \{1, \ldots, d\}>\]
\[S=R[w_{ij}]/<w_{ij}-w_{ji}, w_{dd} \text{ s.t. } i\neq j \in \{1, \ldots, d\}>.\]
Let $\J$ be the defining ideal of $\R(I)$, $\L$ the defining ideal of $\Sym(I)$, and $I(X)$ be the defining ideal of $\F(I)$.
That is, $\R(I)\cong S/ \J$, $\Sym(I)\cong S/ \L$, and $\F(I)\cong W/I(X)$.
Let $\mathcal M=[w_{ij}]$ be the $d\times d $ symmetric matrix with $w_{ij}=w_{ji}$ and $w_{dd}=0$.
\end{data}

As an abuse of notation we may write $W=\mathbb{K}[w_{ij}]$ and $S=R[w_{ij}]$, keeping the understanding that $w_{ij}=w_{ji}$ and $w_{dd}=0$.
In this paper, if the relative size of $i,j$ is known, we typically write the variables as $w_{ij}$ with $i\leq j$.

We note that defining ideal of the symmetric algebra can be easily read from the presentation matrix of the ideal $I$.
Let $g_{ij}=x_ix_j$ if $i\neq j$ and $g_{ii}=x_i^2-x_d^2$ be the minimal generators of $I$, listed above.
Let  $\underline{x}=[x_1, \ldots, x_d]$, $\underline{w}=[w_{11}, \ldots w_{(d-1) d}]$, and $\underline{g}=[g_{11}, \ldots, g_{(d-1)d}]$  be row vectors such that the indices of $\underline{w}, \underline{g}$ are listed in the same order.
Let $\vartheta$ be a minimal homogenous syzygy matrix of $\underline{g}$. 
The defining ideal for the symmetric algebra is well known to be $\L=I_1(\underline{w} \cdot \vartheta)$.

We plan to use information about the special fiber ring of $\m^2$ to better understand the special fiber ring of $I$.
To do this, we will need analogous notation to \Cref{familydata} for $\m^2$.
We also need maps relating the rings corresponding to $I$ and $\m^2$.
\begin{data}[Analogue for $\m^2$]
Let $U= \mathbb{K}[u_{ij}]/< u_{ij}-u_{ji} \text{ s.t. } i\neq j \in \{1, \ldots, d\}>$ be a polynomial ring and let $I(Y)$ be the defining ideal of $\F(\m^2)$.
Let $\mathcal N=[u_{ij}]$ be the $d\times d$ symmetric matrix of variables with $u_{ij}=u_{ji}$.
Define the maps $\epsilon$, $\varphi_U$, and $\varphi_W$ according to the commutative diagram below with $\epsilon(w_{ij})=u_{ij}$ if $i\neq j$, $\epsilon(w_{ii})=u_{ii}-u_{dd}$, $\varphi_U(u_{ij})=x_i x_j$, and $\varphi_W(w_{ij})=g_{ij}$.
Notice that $\varphi_U \circ \epsilon=\varphi_W$.

\begin{center}
\begin{tikzcd}
0\arrow[r] & I(X) \arrow[r] \arrow[d] & W \arrow[d, "\epsilon" , hook] \arrow[r,  "\varphi_W"] & \F(I) \arrow[r] \arrow[d, "i", hook] &0\\
0\arrow[r] & I(Y) \arrow[r] & U\arrow[r,  "\varphi_U"]& \F(\m^2) \arrow[r] &0
\end{tikzcd}
\end{center}
\end{data}

Notice that $\F(\m^2)$ is the coordinate ring of a degree two Veronese variety.
The defining ideal of $\mathcal F(\m^2)$ is well known (for example \cite{Con}, \cite{LP}) to be $I_2( \mathcal N)$.
One might hope that the defining ideal of $\F(I)$ might be defined from $2\times 2$ minors of a similarly constructed matrix, such as $\mathcal M$, which is the preimage of $\mathcal N$ under $\epsilon$. 
While $I(X)\neq I_2(\mathcal M)$, by factoring the map $\varphi_W$ through $\varphi_U$, we can gain insights into the elements of $I(X)$.  
In particular, we will determine a criterion which guarantees that a polynomial $f\in W$ will be an element of $I(X)$, the defining ideal of $\F(I)$. 
Notice that $\mathcal M=\epsilon^{-1}(\mathcal N-u_{dd}\text{I}_\text{d})$ where $\text{I}_\text{d}$ is the $d\times d $ identity matrix.
Indeed, the defining ideal of $\mathcal F(I)$ is given by $I(X)=Ker(\varphi_W)=Ker(\varphi_U\circ \epsilon)=\epsilon^{-1}(Ker(\varphi_U)\cap Im(\epsilon))$which gives rise to the following criteria to check if an element is in $I(X)$.
This criteria will only help us to show our upcoming candidate $\Lambda$ is contained in $I(X)$.
Other techniques will be required to show that these ideals are equal.

\begin{criteria}\label{c}
If $f \in W$ with  $\epsilon(f)\in Ker(\varphi_U)=I_2(\mathcal N)$ then $f\in I(X)=Ker(\varphi_W)$.
\end{criteria}

The elements $f\in W$ that will be explored are linear combinations of the natural generators of $I_2(\mathcal M)$ crafted to satisfy \Cref{c}.
In particular, these elements will depend on the interaction between the minor and the diagonal of $\mathcal M$ and will be described  in \Cref{candidate}.
We will now set the notation for the minors of $\mathcal M, \mathcal N$, taking special care to describe how $2\times 2$ minors interact with the diagonal. 

\begin{definition}\label{def:matrixNotation}
Let $\mathcal B$ be a $d\times d$ symmetric matrix of variables. 
\begin{itemize}
\item Let $m^{\mathcal B}=[a_1 a_2\ldots a_r; b_1b_2\ldots b_c ]^{\mathcal B}=[\underline{ a};\underline{ b}]^{\mathcal B}$ be the submatrix of $\mathcal B$ obtained from selecting rows $\underline{a}=\{a_1,a_2, \ldots, a_r\}$ and columns $\underline{ b}=\{b_1, b_2, \ldots, b_c\}$. Let $\det(m^{\mathcal B})=[\underline{ a}|\underline{ b}]^{\mathcal B}$ be the corresponding minor when $r=c$.
\item If the rows and columns are equal, that is $\underline{ a}=\underline{ b}$, we write $[\underline{ a}; ]^{\mathcal B}$ for the principal submatrix and $[\underline{a}]^{\mathcal B}$ for the corresponding principal minor. 
\item For $s\in \{0,1,2\}$, let $A_s=\{[ij|kl]^{\mathcal M} \text{ s.t. }|\{i,j,k,l\}|=4-s\}$. That is, the subset of the set of $2\times 2$ minors of $\mathcal{M}$ containing of exactly $s$ elements from the diagonal of $\mathcal M$. Equivalently,  $A_s=\{[ij|kl]^{\mathcal M} \text{ s.t. }|\{i,j\}\cap\{k,l\}|=s\}$.
\item We say that the $2\times 2$ submatrices $m^{\mathcal B}$ and $ n^{\mathcal B}$ are \emph{complementary} if there is a $4\times 4$ principal submatrix $[ijkl;]^{\mathcal B}$ containing $m^{\mathcal B}$ and $ n^{\mathcal B}$ where they share no rows and no columns. We say that minors are \emph{complementary} if they are obtained from complementary submatrices.
\end{itemize}
The superscript of $^{\mathcal B}$ may be omitted throughout the paper when clear from context. 
\end{definition}

It is important to note that the sign of a minor obtained from a principal submatrix will always be positive since if we remove column $i$, we also remove row $i$ which will contribute $(-1)^{i+i}=1$ to the sign of the minor.
Similarly, if we are taking a minor of any principal submatrix of $\mathcal{M}$ or $\mathcal{N}$, the sign of that minor can be determined directly from the principal submatrix.
We also observe that the set of diagonal entries of $[\underline{a}; ]^{\mathcal B}$ is a subset of the diagonal entries of $\mathcal B$. 
Notice that  $A_0\sqcup A_1 \sqcup A_2$ partition the set of $2\times 2 $ minors of $\mathcal M$.
We use the sets $A_s$ to distinguish between the $2\times 2$ minors since the diagonal of $\mathcal M$ has different behavior under $\varphi_W$ than the rest of the matrix.

\begin{example}
The underlined and double underlined $2\times 2$ submatrices in (A) are complementary, but are not complementary in (B) because they share a column. The minor obtained from the underlined submatrix in (B) is an element in $A_0$ while the other marked minors are elements of $A_1$.
$$(A):\begin{bmatrix}
w_{11}&\underline{w_{12}}&w_{13}&\underline{w_{14}}\\
\underline{\underline{ w_{12}}}&w_{22}&\underline{\underline{w_{23}}}&w_{24}\\
\underline{\underline{w_{13}}}&w_{23}&\underline{\underline{w_{33}}}&w_{34}\\
w_{14}&\underline{w_{24}}&w_{34}&\underline{w_{44}}
\end{bmatrix}
\hspace{1in}
(B):\begin{bmatrix}
w_{11}&\underline{w_{12}}&\underline{w_{13}}&w_{14}\\
\underline{\underline{ w_{12}}}&w_{22}&\underline{\underline{w_{23}}}&w_{24}\\
\underline{\underline{w_{13}}}&w_{23}&\underline{\underline{w_{33}}}&w_{34}\\
w_{14}&\underline{w_{24}}&\underline{w_{34}}&w_{44}
\end{bmatrix}$$
\end{example}

When $d=4$ all $2\times2$ minors will have unique complements as there is only one $4\times 4$ submatrix of $\mathcal M$, namely $\mathcal M$ itself.
When $d\geq 4$, any $ \det(m^{\mathcal M})\in A_0$ will have a unique complement.
This is because $|\{i,j,k,l\}|=4$ so there is only one $4\times 4$ principal submatrix of $\mathcal M$ containing $m^{\mathcal M}$ as a submatrix.
Alternatively, if $d>4$ and $\det(m^{\mathcal M})\in A_1\cup A_2$, then the complement will not be unique since there will be multiple $4\times 4$ principal submatrices of $\mathcal M$ containing $m^{\mathcal M}$ as a submatrix.
In particular, there is a different complement associated to each principal submatrix of $\mathcal M$ containing $m^{\mathcal M}$.

\section{construction of $\Lambda$, the candidate for $I(X)$}\label{candidate}

We now begin to construct $\Lambda$, the candidate for the defining ideal of $\F(I)$.
We will construct ideals $\Lambda_i$ from the sets $A_i$ and show they are contained in $I(X)$ using  \Cref{c}. 
The candidate $\Lambda$ will be the sum of the ideals $\Lambda_i$. 
We begin by constructing an ideal $\Lambda_0$ which will be generated by minors in $A_0$.
Since minors in $A_0$ do not involve the diagonal, there will be no obstructions to containment in $I(X)$.

\begin{definition}\label{def:L0}
Let $\Lambda_0$ be the $W$-ideal \[\Lambda_0=<[ij|kl]^{\mathcal M} \text{ s.t. } [ij|kl]^{\mathcal M} \in A_0 >.\]
\end{definition}

\begin{lemma}\label{l0}
Adopt  \Cref{familydata} and  \Cref{def:L0}. 
The ideal $\Lambda_0$ is a subideal of $I(X)$, the defining ideal of $\mathcal{F}(I)$.
\end{lemma}

\begin{proof}
Let $[ij|kl]^{\mathcal M}\in A_0$.
By \Cref{c}, we must show that $\epsilon([ij|kl]^{\mathcal M}) \in I_2(\mathcal N)$.
Since none of the variables in $[ij|kl]^{\mathcal M}$ come from the diagonal of $\mathcal M$, the map $\epsilon$ simply will change the $w$'s to $u$'s. 
In particular, $\epsilon([ij|kl]^{\mathcal M})= \epsilon(\det([ij;kl]^{\mathcal M}))=\det( \epsilon([ij;kl]^{\mathcal M}))=\det([ij;kl]^{\mathcal N})=[ij|kl]^{\mathcal N}$.
Since $[ij|kl]^{\mathcal N}\in  I_2(\mathcal N)$, it follows by  \Cref{c} that $[ij|kl]^{\mathcal M}\in I(X)$. In particular, $\Lambda_0\subset I(X)$.\end{proof}

Elements of the candidate $\Lambda$ which involve elements of $A_1, A_2$, will require linear combinations of complementary minors. 
Indeed, we observe that if $m\in A_s$, then its complement is in $A_s$.
This shows that the interactions of minors and the diagonal of $\mathcal M$ is stable with respect to complements.

\begin{lemma}
If $m_1$ and $m_2$ are complementary submatrices, then $\det(m_1)\in A_s$ if and only if $\det(m_2)\in A_s$.
\end{lemma}

\begin{proof}
Let $m_1$ and $m_2$ be $2\times 2$ complementary submatrices.
Let $[ijkl;]$ be the $4\times 4 $ principal submatrix of $\mathcal M$ containing $m_1,m_2$ as submatrices.
Let $R_1,R_2, C_1, C_2$ be the sets of row and column indices for $m_1, m_2$, respectively. 
Since $m_1, m_2$ are complementary, it follows that $\{i,j,k,l\}=R_1\cup R_2=C_1\cup C_2$ and $\emptyset=R_1\cap R_2= C_1\cap C_2$. 

Let $s=|R_1\cup C_1|$ and let $t=|R_2\cup C_2|$. 
Then we have $4-s=|R_1 \cap C_1|$ and $4-t=|R_2\cap C_2|$ by  \Cref{def:matrixNotation}.
Further, $\det(m_1)\in A_{4-s}$ and $\det(m_2)\in A_{4-t}$.
It remains to show that $s=t$.
If  $\{i,j,k,l\} \setminus (R_1 \cup C_1)=R_2\cap C_2$, then $4-s=4-t$ and we're done.

Suppose $a\in R_2\cap C_2$.
Since $R_1\cap R_2=\emptyset$, it follows that $a\not \in R_1$.
Similarly, $a\not \in C_1$.
Therefore $a\not \in R_1\cup C_1$ and  $a\in \{i,j,k,l\} \setminus (R_1 \cup C_1)$.
On the other hand, suppose $a \in \{i,j,k,l\} \setminus (R_1 \cup C_1)$.
Since $a\in \{i,j,k,l\}$, it follows that $a\in (R_1\cup C_1\cup R_2)$. 
But since $a\not\in (R_1 \cup C_1)$, it follows that $a\in R_2$.
Similarly $a\in C_2$ and $a\in R_2\cap C_2$, completing the proof.
\end{proof}

The next ideal $\Lambda_1$ will be constructed from elements of $A_1$. 
In particular, it will be generated by the forms $\det(m_1^{\mathcal M}) \pm \det(m_2^{\mathcal M})$, where $\det(m_1^{\mathcal M}), \det(m_2^{\mathcal M})$ are complementary minors in the set $A_1$.
Notice that the signs of $\det(m_1^{\mathcal M}),\det(m_2^{\mathcal M})$ always agree, so we focus on the locations of the diagonal entries in the $2\times 2$ submatrices to determine whether to add or subtract the complementary minors. 
In particular, after applying the map $\epsilon$, each minor $\det(m_i^{\mathcal M})\in A_1$ will consist of a $2\times 2$ minor in $\mathcal N$ and one additional term. 
This term will be a multiple of the variable $u_{dd}$, whose sign will change based on if it appears in the diagonal or antidiagonal of the corresponding submatrix. 
Balancing two complementary minors will allow this term to cancel and apply  \Cref{c}.

\begin{definition}\label{def:L1}
Let $\det(m^{\mathcal M})\in A_1$ and $m^{\mathcal M}$ its corresponding $2\times 2$ submatrix. 
If $w_{ii}$ be the unique entry from the diagonal of $\mathcal M$ which appears in $m^{\mathcal M}$, we say $i$ the \emph{repeated index} of $\det(m^{\mathcal M})$.
Let $\delta: A_1\rightarrow \{0,1\}$ be the map such that $\delta(\det(m^{\mathcal M}))=\begin{cases} 0 & w_{ii} \text{ is in the diagonal of } m^{\mathcal M}\\ 1 & w_{ii} \text{ is in the antidiagonal of } m^{\mathcal M}\end{cases}$. 

Let $\Lambda_1$ be the $W$-ideal \[\Lambda_1= <\det(m^{\mathcal M})-(-1)^{\delta(m^{\mathcal M})+\delta(n^{\mathcal M})}\det(n^{\mathcal M})\hspace{.05in} \text{ s.t. }  \det(m^{\mathcal M}),\det(n^{\mathcal M})\in A_1 \text{ are complementary }>.\] 
\end{definition}

\begin{lemma}\label{l1}
Adopt  \Cref{familydata} and  \Cref{def:L1}.
The ideal $\Lambda_1$  is a subideal of $I(X)$, the defining ideal of $\mathcal{F}(I)$.
\end{lemma}
\begin{proof}
Let $\det(m^{\mathcal M}),\det(n^{\mathcal M}) \in A_1$ be complementary minors in the $4\times 4$ principal submatrix $[ijkl;]^{\mathcal M}$ where $j$ is the repeated index of $m^{\mathcal M}$ and $l$ is the repeated index of $n ^{\mathcal M}$.
Let $f=\det(m^{\mathcal M})-(-1)^{\delta(m^{\mathcal M})+\delta(n^{\mathcal M})}\det(n^{\mathcal M})$.
By  \Cref{c}, we must show that $\epsilon(f) \in I_2(\mathcal N)$.
First consider the case $i,k<j$.
The image of $m^{\mathcal M}$ under $\epsilon$ will be
\begin{align*}
\epsilon\left(\det\left( \begin{bmatrix} w_{ik}&w_{ij}\\ w_{jk}& w_{jj} \end{bmatrix}\right)\right)&=\det\left(\epsilon\left( \begin{bmatrix} w_{ik}&w_{ij}\\ w_{jk}& w_{jj} \end{bmatrix}\right)\right)=\det\left( \begin{bmatrix} u_{ik}&u_{ij}\\ u_{jk}& u_{jj}-u_{dd} \end{bmatrix}\right)=[ij|kj]^{\mathcal N}-u_{ik}u_{dd}
\end{align*}
Similarly, if $j<i,k$, then $\epsilon(\det(m^{\mathcal M}))=[ji|jk]^{\mathcal N}-u_{ik}u_{dd}=[ij|kj]^{\mathcal N}-u_{ik}u_{dd}$. 
On the other hand, when $i<j<k$ or $k<j<i$, then the repeated index is on the antidiagonal,  and 
$\epsilon(\det(m^{\mathcal M}))=[ij|jk]^{\mathcal N}+u_{ik}u_{dd}$ or $\epsilon(\det(m^{\mathcal M}))=[ji|kj]^{\mathcal N}+u_{ik}u_{dd}$. 
After reordering the rows or columns in these cases, $\epsilon(\det(m^{\mathcal M}))=-[ij|kj]^{\mathcal N}+u_{ik}u_{dd}$.
In all cases, $\epsilon(\det(m^{\mathcal M}))=(-1)^{\delta(m^{\mathcal M})}[ij|kj]^{\mathcal N}-(-1)^{\delta(m^{\mathcal M})}u_{ik}u_{dd}.$
Similarly, $\epsilon(\det(n^{\mathcal M}) ) =(-1)^{\delta(n^{\mathcal M})} [il|kl]^{\mathcal N}-(-1)^{\delta(n^{\mathcal M})}u_{ik}u_{dd}.$
Putting these together, we can compute $\epsilon(f)$ and cancel like terms. 
\begin{align*}
\epsilon(f)=
&\epsilon(\det(m^{\mathcal M})-(-1)^{\delta(m^{\mathcal M})+\delta(n^{\mathcal M})}\det(n^{\mathcal M}))\\
&=\epsilon(\det(m^{\mathcal M}))-(-1)^{\delta(m^{\mathcal M})+\delta(n^{\mathcal M})}\epsilon(\det(n^{\mathcal M}))\\
&=(-1)^{\delta(m)}[ij|kj]^{\mathcal N}-(-1)^{\delta(m)}u_{ik}u_{dd}-(-1)^{\delta(n^{\mathcal M})} [il|kl]^{\mathcal N}+(-1)^{\delta(n^{\mathcal M})}u_{ik}u_{dd}\\
&=(-1)^{\delta(m^{\mathcal M})}[ij|kj]^{\mathcal N}-(-1)^{\delta(n^{\mathcal M})}[il|kl]^{\mathcal N}
\end{align*}
Since $\epsilon(f) \in I_2(\mathcal N)$, by  \Cref{c} we have $f\in I(X)$.
Therefore $\Lambda_1\subset I(X)$, as claimed.
\end{proof}

The final ideal $\Lambda_2$ will be constructed using minors from $A_2$.
Indeed, the obstruction to the sum of complementary minors $\det(m^{\mathcal M})+\det(n^{\mathcal M})$ lying inside $I(X)$ will be related to the trace of $[ijkl;]$, the $4\times 4$ principal submatrix containing $m^{\mathcal M}$ and $n^{\mathcal M}$.
To sidestep this obstruction we will introduce the notion of a pair of complementary minors.
\begin{definition}\label{def:L2}
We say $(m_1^{\mathcal M},n_1^{\mathcal M}),(m_2^{\mathcal M}, n_2^{\mathcal M})$ are a pair of complementary matrices or (PCM) provided all of the following are true:
\begin{enumerate}[(i)]
\item $\det(m_1^{\mathcal M}),\det(n_1^{\mathcal M})\in A_2$ are complementary minors
\item $\det(m_2^{\mathcal M}), \det(n_2^{\mathcal M}) \in A_2$ are complementary minors
\item There is a $4\times 4$ principal submatrix $[ijkl;]$ containing $m_1^{\mathcal M}, n_1^{\mathcal M}, m_2^{\mathcal M}, n_2^{\mathcal M}$ as submatrices.
\end{enumerate}
Let $\Lambda_2$ be the $W$-ideal \[\Lambda_2 =< (\det(m_1^{\mathcal M}) +\det(n_1^{\mathcal M}))-(\det(m_2^{\mathcal M})+\det(n_2^{\mathcal M})) \hspace{.05in} \text{ s.t. } (m_1^{\mathcal M},n_1^{\mathcal M}),(m_2^{\mathcal M},n_2^{\mathcal M}) \text{ are PCM} >.\]
\end{definition}

\begin{lemma}\label{l2}
Adopt  \Cref{familydata} and  \Cref{def:L2}.
The ideal $\Lambda_2 $ is a subideal of $I(X)$, the defining ideal of $\mathcal{F}(I)$.
\end{lemma}
\begin{proof}
We first consider a single minor $[ij]^{\mathcal M}\in A_2$ with $i<j$ and compute its image under $\epsilon$. 
Indeed,
\begin{align*}
\epsilon(\det([ij;]^{\mathcal M}) &=\det(\epsilon([ij;]^{\mathcal M}) =\det\begin{bmatrix} u_{ii}-u_{dd} & u_{ij} \\ u_{ij} & u_{jj}-u_{dd}\end{bmatrix} \\
&=(u_{ii}-u_{dd})(u_{jj}-u_{dd})-u_{ij}^2=u_{dd}^2- (u_{ii}+u_{jj}) (u_{dd})+[ij]^{\mathcal N}
\end{align*}
Let $(m_1^{\mathcal M},n_1^{\mathcal M}),(m_1^{\mathcal M}, n_2^{\mathcal M})$ be a pair of complementary matrices in $[ijkl;]^{\mathcal M}$.
Let $f=(\det(m_1^{\mathcal M}) +\det(n_1^{\mathcal M}))-(\det(m_2^{\mathcal M}) +\det(n_2^{\mathcal M}))$.
We will show that $\epsilon(f)\in I_2(\mathcal N)$. 
Without loss of generality, we may assume $\det(m_1^{\mathcal M})=[ij]^{\mathcal M},\det(n_1^{\mathcal M})=[kl]^{\mathcal M},\det(m_2^{\mathcal M})=[ik]^{\mathcal M}, \det(n_2^{\mathcal M})=[jl]^{\mathcal M}$
Now, 
\begin{align*}
\epsilon(f) &= \epsilon((\det(m_1^{\mathcal M}) +\det(n_1^{\mathcal M}))-(\det(m_2^{\mathcal M}) +\det(n_2^{\mathcal M}))\\
&=\epsilon(\det(m_1^{\mathcal M})) +\epsilon(\det(n_1^{\mathcal M}))-(\epsilon(\det(m_2^{\mathcal M}))+\epsilon(\det(n_2^{\mathcal M})))\\
&=u_{dd}^2- (u_{ii}+u_{jj}) (u_{dd})+[ij]^{\mathcal N}+u_{dd}^2- (u_{kk}+u_{ll}) (u_{dd})+[kl]^{\mathcal N}\\
& \hspace{.2in} -( u_{dd}^2- (u_{ii}+u_{kk}) (u_{dd})+[ik]^{\mathcal N}+u_{dd}^2- (u_{jj}+u_{ll}) (u_{dd})+[jl]^{\mathcal N})\\
&=[ij]^{\mathcal N}+[kl]^{\mathcal N}-[ik]^{\mathcal N}-[jl]^{\mathcal N}
\end{align*}
Since $\epsilon(f) \in I_2(\mathcal N)$, by  \Cref{c} we have $f\in I(X)$.
Therefore $\Lambda_2\subset I(X)$, as claimed.

\end{proof}

\begin{definition}\label{def:lambda}
Let $\Lambda=\Lambda_0+\Lambda_1+\Lambda_2$ where $\Lambda_i$ are the ideals defined in  \Cref{def:L0,def:L1,def:L2}, respectively. 
That is,
\begin{align*}
\Lambda_0 &=< \det(m^{\mathcal M}) \hspace{.05in}\text{s.t.} \hspace{.05in} \det(m^{\mathcal M})\in A_0 >\\
\Lambda_1 &= <\det(m^{\mathcal M})-(-1)^{\delta(m)+\delta(n)}\det(n^{\mathcal M})\hspace{.05in} \text{ s.t. }  \det(m^{\mathcal M}),\det(n^{\mathcal M})\in A_1 \text{ are complementary }>\\
\Lambda_2 &=< (\det(m_1^{\mathcal M}) +\det(n_1^{\mathcal M}))-(\det(m_2^{\mathcal M}) +\det(n_2^{\mathcal M})) \hspace{.05in} \text{ s.t. } (m_1^{\mathcal M},n_1^{\mathcal M}),(m_1^{\mathcal M}, n_2^{\mathcal M}) \text{ are PCM} >
\end{align*}
\end{definition}

\begin{theorem}\label{subset}
Adopt  \Cref{familydata} and  \Cref{def:lambda}.
The ideal $\Lambda$ is a subideal of $I(X)$, the defining ideal of $\mathcal{F}(I)$.
\end{theorem}
\begin{proof}
This statement follows directly from  \Cref{l0,l1,l2}.
\end{proof}

\section{The Fiber Ideal}\label{sec:fiber}
The goal of this section is to compute the generating degree of $I(X)$, the defining ideal of $\F(I)$, along with the Hilbert function of $I(X)$ up to its generating degree. 
To do this, we use two main pieces of information. 
First, since $I$ and $\m^2$ differ by exactly one generator, we can construct a short exact sequence relating $\F(I)$ to $\F(\m^2)$.
Then, since the defining ideal of $\F(\m^2)$ is well known, we use the long exact sequence of Tor to bound the generating degree of the defining ideal for $\F(I)$.
In particular, we compute $[\Tor_1^W(\mathcal F(I),\mathbb{K})]_j$ to show that $I(X)$ is generated in degrees two and three.

\begin{theorem}\label{generatingdegreeofI(X)}
Adopt  \Cref{familydata}.
The graded module,  $[\Tor_1^W(\F(I),\mathbb{K})]_j=0$ for all $j\geq 4$. That is, the ideal $I(X)$ is generated in degrees 2 and 3.
\end{theorem}

\begin{proof}
We begin by setting the degrees of the generators of $I$ and $ \m^2$ to be in degree 1.
Since $I^2=\m^4$, we have the short exact sequence 
\[0\rightarrow \mathcal{F}(I)\rightarrow \mathcal{F}(\m^2)\rightarrow \mathbb{K}(-1)\rightarrow 0.\]
Applying the long exact sequence of Tor to this sequence, we have 
\[\cdots\rightarrow \Tor_2^W(\mathbb{K}(-1),\mathbb{K})\rightarrow \Tor_1^W(\mathcal F(I),\mathbb{K})\rightarrow \Tor_1^W(\mathcal F(\m^2),\mathbb{K}) \rightarrow \cdots . \]
As this is a graded sequence, it is enough to show that for $j\geq 4$ we have $[ \Tor_2^W(\mathbb{K}(-1),\mathbb{K})]_j=0$ and $[\Tor_1^W(\mathcal F(\m^2),\mathbb{K})]_j=0$.

Notice that $\Tor_i^W(\mathbb{K}(-1),\mathbb{K})=\Tor_i^W(\mathbb{K},\mathbb{K})(-1)$.
Since $\Tor_i^W(\mathbb{K},\mathbb{K})$ can be computed with the Koszul complex, we see $[\Tor_i^W(\mathbb{K},\mathbb{K})]_j=0$ for all $j\neq i$.
In the case when $i=2$, we have $[\Tor_2^W(\mathbb{K},\mathbb{K})]_{k}=[\Tor_2^W(\mathbb{K},\mathbb{K})(-1)]_{k+1}=[\Tor_2^W(\mathbb{K}(-1),\mathbb{K})]_{k+1}=0$ for all $k\neq 2$.
In particular, $[\Tor_2^W(\mathbb{K}(-1),\mathbb{K})]_{j}=0$ whenever $j\neq 3$ and therefore when $j\geq 4$.

Now consider $[\Tor_1^W(\mathcal F(\m^2),\mathbb{K})]_j$.
Recall that this graded module is nonzero precisely when $\F(\m^2)$ has a minimal generator in degree $j$ over $W$.
Notably, in the ring $U$, the defining ideal of $\F(\m^2)$ is $I(Y)=I_2(\mathcal N)$, the two by two minors of a symmetric matrix of the variables in $U$.
By using the commutative diagram below, we may view $U$ as a polynomial ring over $W$, that is, $U\cong W[u_{dd}]$.
 \[\xymatrix{W  \ar@{->>}[d]^{\varphi_W} \ar@{^{(}->}[r]^\epsilon &U \ar@{->>}[d]^{\varphi_U}\\ \mathcal F(I) \ar@{^{(}->}[r] &\mathcal F(\m^2)}\]
Notice that $\varphi_U(u_{dd}^2)=x_d^4\in \m^4=I^2$ so there are $a_s\in \mathbb{K}$ and $f_s\in I^2$ such that $\varphi_U(u_{dd}^2)=\sum a_sf_s$.
Since $\varphi_W$ is surjective, we can find degree two elements $g_s\in W$ so that $\varphi_W(g_s)=f_s$.
Let $h=u_{dd}^2-\sum a_sg_s$ and $H=<h>$ the ideal generated by $h$ in $W[u_{dd}]$.
Notice that $u_{dd}$ is integral over $W$ with degree two equation of integrality $h$.
So in particular $\F(\m^2)$ is integral over $\F(I)$ and $ \mathcal F(\m^2)=U/I(Y)$ is a $U/H$ module.
Indeed, by the third isomorphism theorem we have 
 $ \mathcal F(\m^2)=\displaystyle\frac{W[u_{dd}]/H}{I(Y)/H}.$
Since $W[u_{dd}]/H=W\oplus Wu_{dd}$ is free over $W$, we have  $ \F(\m^2)=\displaystyle \frac{W\oplus Wu_{dd}}{I(Y)/H}$.

We will study the presentation matrix $\phi$ of the generating set of $2\times 2 $ minors of $\mathcal N$ in the $W$ resolution of $\F(m^2)$. 
 \[\xymatrix{W\oplus Wu_{dd}  \ar@{->>}[r]^{\phi} & \F(\m^2) }.\]

We define a map $\epsilon^*: U\rightarrow W\oplus Wu_{dd}$ and recall the map $\epsilon:W\rightarrow U$
\[\epsilon(u_{ij})=\begin{cases}  u_{ij} & i\neq j \\ u_{ii}-u_{dd} & (i=j)\neq d \\ u_{dd} & i=j=d
   \end{cases} \hspace{1in} \epsilon(w_{ij})=\begin{cases}  u_{ij} & i\neq j \\ u_{ii}-u_{dd} & i=j 
   \end{cases}\]
We consider three cases based on the image of a $2\times 2$ minor of $\mathcal N$ under $\epsilon^{*}$. Let $i,j,k,l$ be distinct.
\begin{enumerate}[(a)]
\item Consider the $2\times 2$ minor $u_{ij}u_{kl}-u_{ik}u_{jl}$ of $\mathcal N$. Over $W\oplus Wu_{dd}$ the generator will be $w_{ij}w_{kl}-w_{ik}w_{jl}$.
In particular, it will contribute the column $\begin{bmatrix} w_{ij}w_{kl}-w_{ik}w_{jl} \\ 0\end{bmatrix}$ to $\phi$.
\item Consider the $2\times 2$ minor $u_{ij}u_{kk}-u_{ik}u_{jk}$ of $\mathcal N$. Over $W\oplus Wu_{dd}$ the generator will be $w_{ij}(w_{kk}+u_{dd})-w_{ik}w_{jk}= (w_{ij}w_{kk}-w_{ik}w_{jk})+w_{ij}u_{dd}$ if $k\neq d$ 
and $w_{ij}(u_{dd})-w_{id}w_{jd}= (w_{ij}-w_{id}w_{jd})+w_{ij}u_{dd}$ when $k=d$. 
In particular, it will contribute the column $\begin{bmatrix} w_{ij}w_{kk}-w_{ik}w_{jk} \\ w_{ij}\end{bmatrix}$ or $\begin{bmatrix} -w_{ik}w_{jk} \\ w_{ij}\end{bmatrix}$ to $\phi$.
\item Consider the $2\times 2$ minor $u_{ii}u_{jj}-u_{ij}^2$ of $\mathcal N$. Over $W\oplus Wu_{dd}$ the generator will be 
\begin{align*}
(w_{ii}+u_{dd})(w_{jj}+u_{dd})-w_{ij}^2&= w_{ii}w_{jj}+w_{ii}u_{dd}+w_{jj}u_{dd}+u_{dd}^2-w_{ij}^2\\
&=w_{ii}w_{jj}+w_{ii}u_{dd}+w_{jj}u_{dd}+\sum a_s g_s-w_{ij}^2\\
&=(w_{ii}w_{jj}+\sum a_s g_s -w_{ij}^2)+(w_{ii}+w_{jj})u_{dd} \end{align*}
when $i,j\neq d$ and $(\sum a_s g_s -w_{ij}^2)+(w_{ii})u_{dd}$ otherwise.

In particular, it will contribute the column $\begin{bmatrix} w_{ii}w_{jj}+\sum a_s g_s-w_{ij}^2 \\ w_{ii}+w_{jj}\end{bmatrix}$ or $\begin{bmatrix} \sum a_s g_s-w_{ij}^2 \\ w_{ii}\end{bmatrix}$ to $\phi$.
\end{enumerate}
Since each column of $\phi$ has degree 2, we have $[\Tor_1^W(\mathcal F(\m^2,\mathbb{K})]_j=0$ for all $j\neq 2$.
In particular, $[\Tor_1^W(\mathcal F(\m^2),\mathbb{K})]_j=0$ for all $j\neq 2$, and in particular, when $j\geq 4$.

Thus $[\Tor_1^W(\mathcal F(I),\mathbb{K})]_j=0$ for all $j\geq 4$ and $I(X)$ is generated in degrees 2 and 3, as claimed.
\end{proof}

\begin{prop}
The Hilbert function of $I(X)$ in degrees 2 and 3 satisfy the following equalities:
\[HF_{I(X)}(2)=\frac{2(d+2)(d+1)(d)(d-3)}{4!}\]
\[HF_{I(X)}(3)=\frac{14 d^6 + 30 d^5 - 40 d^4 - 210 d^3 - 334 d^2 - 180 d}{6!}.\]
\end{prop}
\begin{proof}
Since $\mathcal{F}(I)=W/I(X)$, we have $HF_{I(X)}(i)=HF_W(i)-HF_{\mathcal{F}(I)}(i)$ for all $i$.
Recall that the Hilbert function of a polynomial ring in $n$ variables in degree $i$ is ${i+n-1}\choose i$. 
Since $W$ has $n={{d+1}\choose 2}-1 $ variables, we may compute the Hilbert function of $W$ in degrees two and three as follows:
 \[HF_W(2)={{2+\left({{d+1}\choose 2} -1\right) -1}\choose 2}={{{{d+1}\choose 2} }\choose 2}=3{{d+2}\choose4}\]
 \[HF_W(3)={{3+{{d+1}\choose 2} -2}\choose 3}={{{{d+1}\choose 2} +1}\choose 3}={{d+2}\choose 4}\Bigg({{d+1}\choose2}+1\Bigg).\]

Since $I^i=\m^{2i}$ for all $i\geq 2$ and ${\mathcal{F}(\m)}$ is isomorphic to $R$, the polynomial ring in $d$ variables,
we may compute the Hilbert function of $\F(I)$ in degrees two and three as follows:
 \[HF_{\mathcal{F}(I)}(2)=\dim_\mathbb{K} \left( \frac{I^2}{\m I^2} \right)=\mu(I^2)=\mu(\m^4)=HF_{\mathcal{F}(\m)}(4)=HF_{R}(4)={d+3\choose 4}\]
  \[HF_{\mathcal{F}(I)}(3)=\dim_\mathbb{K} \left( \frac{I^3}{\m I^3} \right)=\mu(I^3)=\mu(\m^6)=HF_{\mathcal{F}(\m)}(6)=HF_{R}(6)={d+5\choose 6}.\]

This allows us to conclude the claimed equalities:
\begin{align*}
HF_{I(X)}(2) &=HF_W(2)-HF_{\mathcal{F}(I)}(2)=3{{d+2}\choose4}-{{d+3}\choose 4}=\frac{2(d+2)(d+1)(d)(d-3)}{4!}\\
HF_{I(X)}(3) &= HF_W(3)-HF_{\mathcal{F}(I)}(3)={{d+2}\choose 4}\Bigg({{d+1}\choose2}+1\Bigg)-{{d+5}\choose 6}\\
&=\frac{14 d^6 + 30 d^5 - 40 d^4 - 210 d^3 - 334 d^2 - 180 d}{6!}.\end{align*}

\end{proof}

We will show that $\Lambda$ is the defining ideal for the special fiber ring $\mathcal{F}(I)$.  
We have already shown that $\Lambda\subseteq I(X)$ and that $I(X)$ is generated in degrees two and three.
It remains to be shown is that $I(X)$ and $\Lambda$ have the same Hilbert function degrees two and three. 
As this computation is laborious, we include the statement of the Main Theorem here.

\begin{theorem}[Main Theorem]\label{maintheorem1}
The ideal $\Lambda$ is the defining ideal of $\mathcal F(I)$, that is $\Lambda=I(X)$.
 Further $\J=\L+ \Lambda S$ is the defining ideal of $\R(I)$.
\end{theorem}
\begin{proof}
By  \Cref{subset}, we have $\Lambda \subseteq I(X)$.
It remains to show $I(X)\subseteq \Lambda$.
By  \Cref{generatingdegreeofI(X)}, $I(X)$ is generated in degrees two and three.
Therefore, it is enough show that the Hilbert functions of $I(X)$ and $\Lambda$ are equal in degree 2 and 3.
 \Cref{HF2} will show $HF_\Lambda(2)=HF_{I(X)}(2)$ and  \Cref{HF3} will show $HF_\Lambda(3)=HF_{I(X)}(3)$, completing the proof that $\Lambda=I(X)$.
Hence the defining ideal of $\mathcal{F}(I)$ is $\Lambda$.
Since the Rees algebra is fiber type, the defining ideal of $\R(I)$ is $\J=\L+ \Lambda S$.
\end{proof}


\section{Hilbert Function of Candidate in Degree 2}\label{hf2}
In this section, we compute the Hilbert function of $\Lambda$ in degree two and show that it agrees with the previously computed value of $HF_{I(X)}(2)$.
To compute the Hilbert function of $\Lambda$ in degree 2,  we describe sets $K_i$ of monomials which we show are elements of $\init(\Lambda_i)$.
Since we are working in the polynomial ring $W$ and not any quotient ring, a set of monomials of the same degree will be linearly independent precisely when they are distinct. 
By construction, these sets will be disjoint, so the sum $|K_0|+|K_1|+|K_2|$ will provide a lower bound for $HF_{\Lambda}(2)$.
Since $\Lambda \subseteq I(X)$ implies $HF_{\Lambda}(2)\leq HF_{I(X)}(2)$, we will show $|K_0|+|K_1|+|K_2|=HF_{I(X)}(2)$.
This not only allows us to conclude $HF_{\Lambda}(2)\geq HF_{I(X)}(2)$, but also shows that the $K_i$ sets are exactly the degree two terms of the initial ideal of $\Lambda$.
We will use this observation describe the degree three terms of the initial ideal of $\Lambda$ in the following section. 

In this section, if a leading term is known, it will be \underline{marked}. 
Since $w_{ij}=w_{ji}$, if the relative size of $i,j$ is known we will typically write the variable as $w_{ij}$ with $i\leq j$ as it will be convenient for our choice of monomial order on the $w_{ij}$ variables.
We will now define the monomial order and $K_i$ sets.

\begin{definition}\label{orderOmega}
Let $\omega$ be the graded reverse lexicographic order such that $w_{ij}<_{\omega} w_{kl}$ if either $\max\{i,j\}<\max\{k,l\}$ or $\max\{i,j\}=\max\{k,l\}$ and $\min\{i,j\}<\min\{k,l\}$. 
We write $\init(\Lambda)=\init_\omega (\Lambda)$ to represent the initial ideal with respect to $\omega$ of the ideal $\Lambda$. 

For any finite set of monomials $B \subset W$, we let $\max_{\omega} B$ be the monomial $b\in B$ such that for all $b'\in B\setminus \{b\}$ we have $b'<_\omega b$.
Similarly, for any finite set of monomials $B \subset W$, we let $\min_{\omega} B$ be the monomial $b\in B$ such that for all $b'\in B\setminus \{b\}$ we have $b'>_\omega b$.
The subscript $\omega$ may be omitted if clear from context.
Let 
\begin{align*}
K_0&=\{w_{ik}w_{jl}, w_{il}w_{jk}\hspace{.05in} \text{ s.t. }\hspace{.05in} i<j<k<l\}\\
K_1&=\{ \max_{\omega}\{w_{jl}w_{il},w_{ij}w_{ll}\} \hspace{.05in} \text{ s.t. }\hspace{.05in}i\neq j, \hspace{.05in} l> \min\{\{1, \ldots , d\}\setminus \{i,j\}\}  \}\\
K_2&=\{w_{ij}^2 \hspace{.05in} \text{ s.t. }\hspace{.05in}   i\geq2, j\geq4, i\neq j\}
\end{align*}
be sets of degree two monomials in $W$.
\end{definition}

Since the generators of $\Lambda_i$ are constructed from linear combinations of elements of $A_i$, we observe a pattern which is useful for describing the initial ideals of $\Lambda_i$.

 \begin{remark}\label{indices}
Every monomial term of a generator of $\Lambda_i$ requires exactly $4-i$ indices.
\end{remark}

The next three lemmas will show that the sets $K_i$ are contained in the initial ideals of $\Lambda_i$ with respect to the ordering $\omega$.
 We will also count the size of these sets to obtain a lower bound on the Hilbert function of $\Lambda$ in degree two.

\begin{lemma}\label{lem:K0}
Adopt  \Cref{familydata} and  \Cref{def:lambda,orderOmega}.
The set of degree two monomials \\ $K_0\subseteq \init \Lambda_0$ and $HF_{\Lambda_0}(2)\geq |K_0|=2{d\choose 4}$.
\end{lemma}

\begin{proof}
Without loss of generality, we may assume $i<j<k<l$ and consider the principal submatrix $[ijkl;]$.
We will show that for each monomial $\pmb{w}\in K_0$, there is a corresponding $F\in \Lambda_0$ such that $\pmb{w}=\init F$.
Let $F_1=[il|jk]=\underline{ w_{ik}w_{jl}}-w_{ij}w_{kl}$  and $F_2=[ik|jl]=\underline{ w_{il}w_{jk}}-w_{ij}w_{kl}$.
Since $F_1,F_2\in \Lambda_0$ with $w_{ik}w_{jl}=\init(F_1)$ and $w_{il}w_{jk}=\init(F_2)$ it follows that $K_0\subseteq \init \Lambda_0$ and $HF_{\Lambda_0}(2)\geq |K_0|$.

For counting, note that each choice of indices $i,j,k,l$ produces two monomials in $K_0$.
Further, each monomial in $K_0$ involves all 4 indices $i,j,k,l$, there cannot be leading terms in common between distinct index choices. 
Since there are ${d\choose 4}$ ways to choose $i,j,k,l$, we have $2{d\choose 4}$ monomials in  $K_0$ and $HF_{\Lambda_0}(2)\geq |K_0|=2{d\choose 4}$, as claimed.
\end{proof}

\begin{lemma}\label{lem:K1}
Adopt  \Cref{familydata} and  \Cref{def:lambda,orderOmega}.
The set of degree two monomials\\ $K_1\subseteq \init \Lambda_1$ and $HF_{\Lambda_1}(2)\geq |K_1|=(d-3){d\choose 2}$.
\end{lemma}

\begin{proof}
We will show that for each monomial $\pmb{w}\in K_1$, there is a corresponding $F\in \Lambda_1$ such that $\pmb{w}=\init F$.
Recall that the generators of $\Lambda_1$ are of the form $\det(m) \pm \det(n)$.
Since we are concerned with determining the leading monomial, we will not compute the sign for the sake of this proof. 
Without loss of generality, we will assume $m=[ik;jk],n=[il;jl]$ with $i<j$, $k<l$.
Let $F=\det(m)\pm \det(n)=(w_{ij}w_{kk}-w_{jk}w_{ik}) \pm (w_{ij}w_{ll}- w_{jl}w_{il})$. 
Notice $w_{ij}w_{kk}<w_{ij}w_{ll} $ and $w_{ik}w_{jk}<w_{il}w_{jl}$.
In particular, the leading monomial of $F$ will be $w_{ij}w_{ll}$ when $j>l$ and  $w_{il}w_{jl}$ when $j<l$.

The condition that $l> \min\{\{1, \ldots , d\}\setminus \{i,j\}\}$ in the definition of $K_1$ ensures that there is an index $k\neq i,j$ which is smaller than $l$ and can serve as the smaller repeated index.
It now follows that $K_1\subseteq \init \Lambda_1$ and $HF_{\Lambda_1}(2)\geq |K_1|$.

For the counting argument, we fix the two non repeated indices, $i\neq j$ and show that for each choice of  $i,j$ are at least $(d-3)$ distinct elements of $K_1$. 
Let $k'=\min\{\{1, \ldots , d\}\setminus \{i,j\}\}$ and let $l'\in\{\{1, \ldots , d\}\setminus \{i,j,k'\}\}$.
Either $w_{ij}w_{l'l'}$ or $w_{jl'}w_{il'}$ will be an element of $K_1$, but not both.
If $l'=d$ the leading term will be $w_{jd}w_{id}$, so the conclusion is unchanged, even after setting $w_{dd}=0$. 
Since $|\{1, \ldots , d\}\setminus \{i,j,k'\}|=d-3$, it follows that $|K_1|=(d-3){d\choose 2}$ and $HF_{\Lambda_1}(2)\geq |K_1|=(d-3){d\choose 2}$.

\end{proof}

\begin{lemma}\label{lem:K2}
Adopt  \Cref{familydata} and  \Cref{def:lambda,orderOmega}.
The set of degree two monomials\\ $K_2\subseteq \init \Lambda_2$ and $HF_{\Lambda_2}(2)\geq |K_2|=\frac{d(d-3)}{2}$.
\end{lemma}

\begin{proof}
We will show that for each monomial $\pmb{w}\in K_2$, there is a corresponding $F\in \Lambda_2$ such that $\pmb{w}=\init F$.
Let $\pmb{w}=w_{ij}^2$ with $i\geq 2, j\geq 4$ and $i\neq j$. 
We assume without loss of generality that $i<j$. 
Let $k=\min\{2, \dots, d\}\setminus \{i,j\}$.
Notice that $k$ will be 2 or 3 so it follows that $k<j$, but we could have either $k<i$ or $i<k$.
Consider the polynomial
\begin{align*}
F&=([1j]+[ik])-([1k]+[ij])\\
&=w_{11}w_{jj}-w_{1j}^2+w_{ii}w_{kk}-w_{ik}^2 -w_{11}w_{kk}+w_{1k}^2-w_{ii}w_{jj}+\underline{w_{ij}^2}.\\
\end{align*}
Since $F\in \Lambda_2$ with $\init{F}=w_{ij}^2$, we see $K_2\subseteq \init \Lambda_2$ and $HF_{\Lambda_2}(2)\geq |K_2|$.
If $j=d$ the proof will be unchanged except for replacing the variable $w_{jj}=w_{dd}$ with a zero which will, importantly, not change the leading term of $F$.

For counting, we assume without loss of generality that $i<j$.
If we fix the larger index $j$ there will be $j-2$ distinct possibilities for a valid index $i\geq 2$. 
In particular, there are $2+3+4+\cdots+(d-2)=\frac{d(d-3)}{2}$ choices of $i,j$.
Therefore $|K_2|= \frac{d(d-3)}{2}$ and $HF_{\Lambda_2}(2)\geq |K_2|=\frac{d(d-3)}{2}$.
\end{proof}

Since $\Lambda_i\subset \Lambda$, it follows that $K_i\subset \init \Lambda$. 
Since the ideals $K_i$ are disjoint, it follows that the sum of $|K_i|$ provides a lower bound on the Hilbert function of $\Lambda$ in degree 2. In fact, we will see that the sum of $|K_i|$ is actually equal to $HF_{I(X)}(2)$, an upper bound on the Hilbert function of $\Lambda$ in degree 2.

\begin{theorem}\label{HF2}
The Hilbert functions of $\Lambda$ and $I(X)$ agree in degree two, that is $HF_{\Lambda}(2)=HF_{I(X)}(2)$.
In particular, the degree two terms of the initial ideal of $\Lambda$ is the set $K_0\cup K_1\cup K_2$. 
\end{theorem}
\begin{proof}
Since $\Lambda\subseteq I(X)$ we have $HF_{\Lambda}(2)\leq HF_{I(X)}(2)$.
For the converse, notice that by  \Cref{indices} the generators of $\init(\Lambda_i)$ are disjoint.
In particular, the sets $K_i$ are disjoint. 
Therefore $HF_{\Lambda}(2) \geq HF_{\Lambda_0}(2)+HF_{\Lambda_1}( 2) + HF_{\Lambda_2}(2)$. 
By  \Cref{lem:K0,lem:K1,lem:K2} we have $HF_{\Lambda_0}(2)+HF_{\Lambda_1}( 2) + HF_{\Lambda_2}(2)\geq |K_0|+|K_1|+|K_2|= 2{d\choose 4} +\frac{d(d-1)(d-3)}{2}+\frac{d(d-3)}{2}$. 
After algebraic manipulations, we see that $2{d\choose 4} +\frac{d(d-1)(d-3)}{2}+\frac{d(d-3)}{2}=\frac{2(d+2)(d+1)(d)(d-3)}{4!}=HF_{I(X)}(2)$.
Now we have $HF_{\Lambda}(2)\geq HF_{I(X)}(2)$ and the Hilbert functions of $I(X)$ and $\Lambda$ must agree in degree 2.

Since the set $K_0\cup K_1\cup K_2$ is contained in $\init(\Lambda)$ and satisfies $|K_0\cup K_1\cup K_2|=HF_{\Lambda}(2)$, we can conclude that the degree two terms of the initial ideal of $\Lambda$ with respect to the ordering $\omega$ is the set $K_0\cup K_1\cup K_2$.
\end{proof}

\section{Hilbert Function of candidate in Degree 3}\label{hf3}
Now that we have shown the Hilbert functions of $\Lambda$ and $I(X)$ agree in degree 2, we must show they agree in degree 3.
Since $I(X)$ is generated in degrees two and three this will complete the proof of the main theorem showing $\Lambda=I(X)$, the defining ideal of $\F(I)$.
Calculating the Hilbert function of $\Lambda$ in degree 3 is more involved than the degree two computation. 
We must carefully calculate the multiples of degree 2 elements of the initial ideal of $\Lambda$.
Additionally, we must calculate those degree 3 monomials in $\init(\Lambda)$ which are not multiples of degree two elements.
While we will not be calculating a minimal Gr\"obner basis for $\Lambda$, the elements we study were originally discovered through S-polynomial computations. 
Care must be taken at all times to avoid double counting, as we are looking for a lower bound on the Hilbert function of $\Lambda$.
Again, degree three monomials will be linearly independent in this setting as long as they are distinct.

Consider $\init(\Lambda)=\oplus [\init(\Lambda)]_i\subseteq \init(I(X))\subset W$ with $[\init(\Lambda)]_i$ in degree $i$, 
we must compute the elements $[\init(\Lambda)]_2=K_0\cup K_1\cup K_2$ generates inside $[\init(\Lambda)]_3$.
We must be careful not to double count elements. 
For example, the monomial $w_{24}^2w_{34}\in [\init(\Lambda)]_3$ has degree two divisors $w_{24}^2\in K_2$ and $w_{24}w_{34}\in K_1$, so the $K_i$ partition introduced in the previous section is dangerous with respect to double counting. 
We will now introduce a more convenient partition of $[\init(\Lambda)]_2$. 
The sets $S_{ij}$ defined below distinguish elements of $[\init(\Lambda)]_2$ based on their largest variable divisor.

\begin{definition}\label{def:TandS}
Adopt  \Cref{familydata}.
Let $S_{ij}= \{w_{kl} \text{ s.t. } w_{kl}w_{ij} \in [\init(\Lambda)]_2 \text{ with } w_{kl}\leq_{\omega} w_{ij}\}$ and let $|S_{ij}|$ be the size of the set $S_{ij}$.
Let $T=\{w_{ij} \pmb{w} \text{ s.t. } \pmb{w}\in [\init(\Lambda)]_2 \text{ and } w_{ij}\in W\}$.
Let 
\begin{align*}T_{ij}^{kl}&=\{w_{ab} w_{kl} w_{ij} \text{ s.t. }w_{kl}\in S_{ij}  \text{ and either } \\
& \hspace{1in}w_{ab}\leq_\omega w_{kl}, \text{ or } w_{kl}<_\omega w_{ab}\leq_\omega w_{ij} \text{ and } w_{ab}\not\in S_{ij} ,\text{ or } w_{ij},w_{kl}\not \in S_{ab} \}.\end{align*}
Let $\tau$ be an ordering of sets $T_{ij}^{kl}\subset T$ so that $T_{ab}^{cd}>_\tau T_{ij}^{kl}$ provided either $w_{ab}>_\omega w_{ij}$ or $w_{ab}=_\omega w_{ij}$ and $w_{cd}>_\omega w_{kl}$.
Let $T_{ij}^{\max}$ be the set $T_{ij}^{kl}$ such that $w_{kl}= \max_\omega S_{ij}$, that is, $w_{kl} $ is the maximum variable in the set $S_{ij}$ with respect to the monomial ordering $\omega$. Similarly, let $T_{ij}^{\min}$ be the set $T_{ij}^{kl}$ such that $w_{kl}=\min_\omega S_{ij}$, that is, $w_{kl} $ is the minimum variable in the set $S_{ij}$ with respect to the monomial ordering $\omega$.
Let $|T_{ij}^{kl}|$ be the size of the set $T_{ij}^{kl}$.
\end{definition}

Notice that to define the set $T_{ij}^{kl}$, we require $w_{kl}\in S_{ij}$. 
This implies, by definition, that $w_{kl}\leq_\omega w_{ij}$.
In particular, the total ordering $\tau$ is essentially recovering the lexicographic ordering on the degree two monomials in $K_0\cup K_1\cup K_2$.
Further, a degree three monomial $\pmb{w}\in T_{ij}^{kl}$ exactly when $ w_{ij}w_{kl}|\pmb{w}$ and there is no $T_{i'j'}^{k'l'}>_\tau T_{ij}^{kl}$ such that $\pmb{w}\in T_{i'j'}^{k'l'}$.

\begin{example}
To help understand  \Cref{def:TandS} above, here is a table for the case $d=4$. 
\begin{center}
\begin{tabular}{|c|c|c|c|}
\hline
 $\pmb{w} \in [\init(\Lambda)]_2$ & $w_{kl}\in S_{ij}$ & $w_{ab}$ such that  $w_{ab} w_{kl} w_{ij}\in T_{ij}^{kl}$&$|T_{ij}^{kl}|$\\
\hline \hline
$w_{34}^2$ & $w_{34}\in S_{34}$ & $w_{11}, w_{12}, w_{22}, w_{13}, w_{23}, w_{33},w_{14} , w_{24}, w_{34}$&9\\
\hline
$w_{24}w_{34}$ & $w_{24}\in S_{34}$ & $w_{11}, w_{12}, w_{22}, w_{13}, w_{23}, w_{33},w_{14} , w_{24}$&8\\
\hline
$w_{14}w_{34}$ & $w_{14}\in S_{34}$ & $w_{11}, w_{12}, w_{22}, w_{13}, w_{23}, w_{33},w_{14} $&7\\
\hline \hline
$w_{24}^2$ & $w_{24}\in S_{24}$ & $w_{11}, w_{12}, w_{22}, w_{13}, w_{23}, w_{33},w_{14} , w_{24}$&8\\
\hline
$w_{14}w_{24}$ & $w_{14}\in S_{24}$ & $w_{11}, w_{12}, w_{22}, w_{13}, w_{23}, w_{33},w_{14} $&7\\
\hline
$w_{33}w_{24}$ & $w_{33}\in S_{24}$ & $w_{11}, w_{12}, w_{22}, w_{13}, w_{23}, w_{33} $&6\\
\hline
$w_{23}w_{24}$ & $w_{23}\in S_{24}$ & $w_{11}, w_{12}, w_{22}, w_{13}, w_{23} $&5\\
\hline
$w_{13}w_{24}$ & $w_{13}\in S_{24}$ & $w_{11}, w_{12}, w_{22}, w_{13} $&4\\
\hline \hline
$w_{33}w_{14}$ & $w_{33}\in S_{14}$ & $w_{11}, w_{12}, w_{22}, w_{13}, w_{23}, w_{33},w_{14} $&7\\
\hline
$w_{23}w_{14}$ & $w_{23}\in S_{14}$ & $w_{11}, w_{12}, w_{22}, w_{13}, w_{23}, w_{33}$&6\\
\hline
\end{tabular}
\end{center}

In this case, $|S_{34}|=3$, $|S_{24}|=5$ and $|S_{14}|=2$ and all other $|S_{ij}|=0$. 
Also we have $|T_{34}^\max|=9$, $|T^\max_{24}|=8$ and $|T^\max_{14}|=7$ and all other $|T^\max_{ij}|=0$.
In particular, 67 is a lower bound for $|T|$ when $d=4$. 

\end{example}

We will be interested in a lower bound for $|T|$.
Notice that $T_{ij}^{kl}\subset T$, for all $w_{ij}w_{kl}\in [\init(\Lambda)]_2$ with $w_{kl}\leq w_{ij}$.
As these sets $T_{ij}^{kl}$ are disjoint by construction, $|T|\geq  \sum|T_{ij}^{kl}|$. 
Recall that the degree three monomials will be linearly independent provided they are distinct, so we will be bounding $|T|$ below by the number of distinct monomials in the sets $T_{ij}^{kl}$.
Our goal is to describe the numbers $|T_{ij}^\max|$ and $|S_{ij}|$ properly.
We will then show that for any fixed $i,j$, the partial sum $\displaystyle \sum_{w_{kl}\in S_{ij}} T_{ij}^{kl}$ can be expressed as a sum of consecutive integers described in terms of $|T^\max_{ij}|$ and $|S_{ij}|$. 
Namely, 
\begin{align*}
\sum_{w_{kl}\in S_{ij}} T_{ij}^{kl}&=(|T^\max_{ij}|-|S_{ij}|+1)+(|T^\max_{ij}|-|S_{ij}|+2)+\cdots+|T^\max_{ij}|\\
&=\frac{1}{2}(|S_{ij}|)(|T^\max_{ij}|+(|T^\max_{ij}|-|S_{ij}|+1))=\frac{1}{2}(|S_{ij}|)(2|T^\max_{ij}|-|S_{ij}|+1).
\end{align*}

\begin{criteria}\label{antidiagonal}
Since $\omega$ is a graded reverse lexicographic ordering compatible with a symmetric matrix of variables, the leading term of any generator of $\Lambda$ must be an antidiagonal of the matrix $\mathcal{M}$.
In particular, if $w_{kl}\in S_{ij}$ then $w_{ij}w_{kl}$ is an antidiagonal of $\mathcal{M}$.
\end{criteria}

We will begin by showing that certain sets $S_{ij}$ will always be empty.
We expect this to be the case since if  $S_{12},S_{13},S_{23}$ were non-empty, then $w_{12},w_{13},w_{23}$, respectively, would need to be the largest variable divisor for an element in $[\init(\Lambda)]_2$.
However, since these variables are very small with respect to $\omega$, there are few candidates for containment in $S_{12},S_{13},S_{23}$ to rule out.
On the other hand, if $S_{ii}$ were to be non-empty, there are again few candidates for containment since the corresponding monomial in $[\init(\Lambda)]_2$ could only come from $K_1$.

\begin{lemma}\label{S_empty}
Adopt  \Cref{def:TandS}.
The sets $S_{12},S_{13},S_{23},$ and $S_{ii}$ for $i=1, \ldots, d$ are empty.
\end{lemma}

\begin{proof}
Suppose, for contradiction, $w_{kl}\in S_{12}$.
Then $w_{kl}\leq_\omega w_{12}$ and $w_{12}w_{kl}$ is an antidiagonal of $\mathcal M$ by  \Cref{antidiagonal}.
In particular,  $S_{12}\subseteq \{w_{12}\}$. 
But since $w_{12}^2\not\in [\init(\Lambda)]_2$, it follows that$S_{12}=\emptyset$.

Similarly, suppose for contradiction $w_{kl}\in S_{13}$.
Then $w_{kl}\leq_\omega w_{13}$, and $w_{13}w_{kl}$ is an antidiagonal of $\mathcal M$ by  \Cref{antidiagonal}.
In particular,  $S_{13}\subseteq \{w_{12}, w_{22}, w_{13}\}$.
But since $w_{12}w_{13},w_{22}w_{13},w_{13}^2\not\in [\init(\Lambda]_2$, it follows that $S_{13}=\emptyset$.

Next, suppose for contradiction $w_{kl}\in S_{23}$.
Then $w_{kl}\leq_\omega w_{23}$, and $w_{23}w_{kl}$ is an antidiagonal of $\mathcal M$ by  \Cref{antidiagonal}.
In particular, $S_{23}\subseteq \{w_{13}, w_{23}\}$.
But since $w_{13}w_{23},w_{23}^2\not\in [\init(\Lambda]_2$, it follows that $S_{23}=\emptyset$.

Notice that $w_{11}$ is not a part of an antidiagonal, so $S_{11}=\emptyset$.
Let $i>1$ and suppose for contradiction $w_{kl}\in S_{ii}$.
Then $w_{kl}\leq_\omega w_{ii}$, and $w_{ii}w_{kl}$ is an antidiagonal of $\mathcal M$ by  \Cref{antidiagonal}.
Antidiagonals of the symmetric matrix $\mathcal M$ involving $w_{ii}$ are of the form $w_{ii}w_{kl}$ with $k<i<l$ . 
No such $w_{kl}$ satisfies $w_{kl}\leq_\omega w_{ii}$.
In particular, $S_{ii}=\emptyset$.
\end{proof}

We will now continue to sets $S_{ij}$ that are non-empty and count the size of these sets.
We assume without loss of generality that $i<j$ and study the case $i=1$ separately.

\begin{lemma}\label{S1j}
Adopt  \Cref{def:TandS}. Let $j\geq 4$ and $k\leq l$.
We have $w_{kl}\in S_{1j}$ if and only if $1<k$ and $2<l<j$.
In particular, $|S_{1j}|=\frac{j(j-3)}{2}$.
\end{lemma}

\begin{proof}
Suppose $1<k\leq l$ and $2<l<j$. We will show $w_{kl}\in S_{1j}$. 
Notice that $w_{kl}<_{\omega} w_{1j}$ and $w_{1j}w_{kl}$ is an antidiagonal of $\mathcal M$.
That is,  $w_{kl}$ is a candidate for inclusion in $S_{1j}$.
If $k=l>2$, then $w_{kl}w_{1j}=w_{kk}w_{1j}\in K_1$ by  \Cref{lem:K1}.
Otherwise $1<k <l<j$. Then by  \Cref{lem:K0}, it follows that $w_{kl}w_{1j}\in K_0$.
In all cases, $w_{kl}\in S_{1j}$.

We will now show that if $w_{kl}$ does not satisfy $1<k$ and $2<l<j$, then $w_{kl}\not \in S_{1j}$. 
We have several cases to consider.
\begin{enumerate}[(a)]
\item If $l>j$, then $w_{kl}\geq_\omega w_{1j}$ and $w_{kl}\not \in S_{1j}$. 
\item If $l=j$, in order to satisfy $w_{kl}\leq_\omega w_{1j}$, we would need $k=1$. 
However, $w_{1j}^2\not \in [\init(\Lambda]_2$, so $w_{kl}=w_{1j}\not \in S_{1j}$. 
\item If $k=1$ and $l<j$, then $w_{1l}<_\omega w_{1j}$.
However, $w_{1l}w_{1j}\not \in  [\init(\Lambda]_2$, so $w_{kl}=w_{1l}\not \in S_{1j}$. 
\item If $2=k=l<j$, notice $w_{kl}=w_{22}<_\omega w_{1j}$. However $w_{22}w_{1j}\not \in  [\init(\Lambda]_2$.
Therefore $w_{22} \not \in S_{1j}$. 
\end{enumerate}
In particular, any monomials that do not satisfy $1<k$ and $2<l<j$ cannot be elements of $S_{1j}$.

Now we must count the size of the set $S_{1j}$,  taking care not to double count since $w_{kl}=w_{lk}$.
In particular, $\displaystyle |S_{1j}|=\sum_{l=3}^{j-1} (\text{ \# $k$ options for } w_{kl})=\sum_{l=3}^{j-1} (l-1)=2+3+\cdots +(j-2)=\frac{j(j-3)}{2}$.
\end{proof}

\begin{lemma}\label{Sij}
Adopt  \Cref{def:TandS}. Let $1<i<j\leq d$, $j\geq 4$, and $k\leq l$.
We have $w_{kl}\in S_{ij}$ if and only if either $i<l<j$ or $l=j$ and $k\leq i$.
In particular $|S_{ij}|=\frac{(j-i)(i+j-1)}{2}$.
\end{lemma}

\begin{proof}
Suppose $i<l<j$, we will show $w_{kl}\in S_{ij}$.
Notice  $w_{kl}<_\omega w_{ij}$ and  $w_{kl}w_{ij}$ is an antidiagonal of $\mathcal M$.
That is,  $w_{kl}$ is a candidate for inclusion in $S_{ij}$.
If $k\neq l$, then $w_{kl}w_{ij}\in K_0$ by  \Cref{lem:K0}.
If $k=l$, then $w_{kl}w_{ij}=w_{kk}w_{ij}\in K_1$ by  \Cref{lem:K1}.
In all cases, $w_{kl}\in S_{ij}$.

Suppose $l=j$ and $k\leq i$, we will show $w_{kl}\in S_{ij}$.
Notice  $w_{kl}<_\omega w_{ij}$ and  $w_{kl}w_{ij}$ is an antidiagonal of $\mathcal M$.
If $k=i$, then $w_{kl}w_{ij}=w_{ij}^2\in K_2$ by  \Cref{lem:K2}.
If $k<i$, then $w_{kl}w_{ij}=w_{kj}w_{ij}\in K_1$ by  \Cref{lem:K1}.
In all cases, $w_{kl}\in S_{ij}$.

We will now show that if $k,l$ do not satisfy $i<l<j$ or $l=j$ and $k\leq i$, then $w_{kl} \not\in S_{ij}$.
If $l>j$ or $l=j$ and $k> i$, then $w_{kl}>_\omega w_{ij}$, and $w_{kl}\not\in S_{ij}$.
Otherwise, $k\leq l\leq i<j$. 
Then $w_{kl}w_{ij}$ cannot form an antidiagonal in $\mathcal M$ and by  \Cref{antidiagonal},  $w_{kl}\not\in S_{ij}$.
In all cases, $w_{kl}\not\in S_{ij}$.

Now we must count the size of the set $S_{ij}$.
Since $w_{kl}=w_{lk}$, we must be careful not to double count. 
In particular, $\displaystyle |S_{ij}|=\sum_{l=i+1}^{j} (\text{ \# $k$ options for } w_{kl})=\left(\sum_{l=i+1}^{j-1} l \right)+ (i)=\frac{(j-i)(i+j-1)}{2}$.

\end{proof}

Now that we have counted $|S_{ij}|$ we must also count $|T^\max_{ij}|$.
Recall $ T^{\max}_{ij}$ is maximal by $\tau$ among all sets of the form $T_{ij}^{*}$. 
The size of this set is denoted $| T^{\max}_{ij}|$.
We will also show that the size of consecutive sets $T^{kl}_{ij}$ with the same subscript differs by exactly one. 
This will provide a lower bound on the the size of $T=\{w_{ab} \pmb{w} \text{ s.t. } \pmb{w}\in [\init(\Lambda)]_2\}$, that is the degree three multiples of degree two elements in the initial ideal of $\Lambda$.

%
%
%
%
\begin{lemma}\label{Tmax}
Adopt  \Cref{def:TandS}. Let $1< i<j$ and $j\geq 4$. 
Then $ T^{\max}_{1j}=T_{1j}^{(j-1)(j-1)}$ and $ T^{\max}_{ij}=T_{ij}^{ij}$.
More explicitly, 
\begin{align*}
 T^{\max}_{1j}&=\{w_{ab}w_{(j-1)(j-1)}w_{1j} \hspace{.05in} \text{s.t. either } \hspace{.05in} w_{ab}\leq_\omega w_{1j} \text{ or } j\leq \min\{a,b\}\}.\\
 T^{\max}_{ij}&=\{w_{ab}w_{ij}^2 \hspace{.05in} \text{s.t. either } \hspace{.05in} w_{ab}\leq_\omega w_{ij} \text{ or } j\leq \min\{a,b\}\}.\\
\end{align*}
\end{lemma}

\begin{proof}
By  \Cref{S1j}, we have $w_{kl}\in S_{1j}$ with $k\leq l$ if and only if $1<k$ and $2<l<j$.
In particular, $\max_\omega S_{1j}=w_{(j-1)(j-1)}$ and $ T^{\max}_{1j}=T_{1j}^{(j-1)(j-1)}$.
Similarly, by  \Cref{Sij}, we have $w_{kl}\in S_{ij}$ if and only if either $i<l<j$ or $l=j$ and $k\leq i$.
In particular, $\max_\omega S_{ij}=w_{ij}$ and $ T^{\max}_{ij}=T_{ij}^{ij}$.

Notice that there is no variable $w_{kl}$ such that $w_{(j-1)(j-1)}<_\omega w_{kl}<_\omega w_{1j}$. Now, by  \Cref{def:TandS}, 
\begin{align*}
T_{1j}^{(j-1)(j-1)}&=\{w_{ab} w_{(j-1)(j-1)} w_{1j} \text{ s.t. } \text{either } w_{ab}\leq_\omega w_{(j-1)(j-1)}, \\ &\hspace{1in}\text{ or }  w_{(j-1)(j-1)}<_\omega w_{ab} \leq_\omega w_{1j} \text{ and } w_{ab}\not\in S_{1j} ,\text{ or } w_{1j},w_{(j-1)(j-1)}\not \in S_{ab} \}\\
&=\{w_{ab} w_{(j-1)(j-1)} w_{1j} \text{ s.t. } \text{either } w_{ab}\leq_\omega w_{1j},\text{ or } w_{1j},w_{(j-1)(j-1)}\not \in S_{ab} \}\\
T_{ij}^{ij}&=\{w_{ab} w_{ij}^2 \text{ s.t. } \text{either } w_{ab}\leq_\omega w_{ij}, \text{ or } w_{ij}<_\omega w_{ab}\leq_\omega w_{ij} \text{ and } w_{ab}\not\in S_{ij} ,\text{ or } w_{ij}\not \in S_{ab} \}\\
&=\{w_{ab} w_{ij}^2 \text{ s.t. } \text{either } w_{ab}\leq_\omega w_{ij},\text{ or } w_{ij}\not \in S_{ab} \}\\
\end{align*}

We now show that when $w_{ab}>_\omega w_{1j}$, the condition $j\leq \min\{a,b\}$ is equivalent to $w_{1j},w_{(j-1)(j-1)}$ both not in $ S_{ab}$.
We assume without loss of generality that $a\leq b$.
If $a=b$, then $S_{ab}=S_{bb}=\emptyset$ by  \Cref{S_empty}.
Otherwise $j\leq a<b$, by  \Cref{Sij} we have $w_{(j-1)(j-1)},w_{1j}\not \in S_{ab}$.
In both of these cases, $j\leq a=\min\{a,b\}$ and $\pmb{w}\in T^{\max}_{1j}$.
We will now show that if  $a<j$, then $\pmb{w}\not \in T^{\max}_{1j}$.
Since $w_{ab}>_\omega w_{1j}$, it follows that $a<j\leq b$.
If $a=1$, by  \Cref{S1j} we have $w_{(j-1)(j-1)}\in S_{1b}$.
If $a>1$, by  \Cref{Sij} we have $w_{1j}\in S_{ab}$.
In both of these cases, $ a<j$ and $\pmb{w}\not \in T^{\max}_{1j}$.

We now show that when $w_{ab}>_\omega w_{ij}$, the condition $j\leq \min\{a,b\}$ is equivalent to $w_{ij}\not\in S_{ab}$.
We assume without loss of generality that $a\leq b$.
If $a=b$, then $S_{ab}=S_{bb}=\emptyset$ by  \Cref{S_empty} and so $\pmb{w}\in T^{\max}_{ij}$.
If $j\leq a<b$, by  \Cref{Sij} we have $w_{ij}\not \in S_{ab}$ and so $\pmb{w}\in T^{\max}_{ij}$.
In both of these cases, $j\leq a$ and $\pmb{w}\in T^{\max}_{ij}$.
We will now show that if  $a<j$, then $\pmb{w}\not \in T^{\max}_{ij}$.
Since $w_{ab}>_\omega w_{ij}$, it follows that $a<j\leq b$.
If $a=1$, by  \Cref{S1j} we have $w_{ij}\in S_{1b}$, so $\pmb{w} \not \in  T^{\max}_{ij}. $
If $a>1$, by  \Cref{Sij} we have $w_{ij}\in S_{ab}$,  so  $\pmb{w} \not \in  T^{\max}_{ij}$.
In both of these cases, $ a<j$ and $\pmb{w}\not \in T^{\max}_{ij}$.

\end{proof}

\begin{corollary}\label{sizeT}
Adopt  \Cref{def:TandS}. Let $1\leq i<j$ and $j\geq 4$.
The size of the set $ T^{\max}_{ij}$ is \[|T^\max_{ij}|=\frac{d^2-2dj+3d+2j^2-4j+2i}{2}.\]
\end{corollary}

\begin{proof}
Assume without loss of generality that $a\leq b$.
Notice that $\displaystyle |T^\max_{ij}|=\sum_{b=1}^d (\text{\# $a$ options for }w_{ab})$.
When $b<j$ then $w_{ab}\leq_\omega w_{ij}$, then to satisfy $a\leq b$ we must have $a\in \{1, 2, \ldots, b\}$, in particular there are $b$ options for $a$. 
For $b=j$, we must have $a\in \{1,2, \ldots i\}$ to satisfy $w_{ab}\leq_\omega w_{ij}$, however, we could also have $a=b=j$ and which would satisfy $j\leq \min\{a,b\}$, in particular there are $i+1$ options for $a$.
For $j<b<d$, we must have $a\in \{j, \ldots b\}$, in particular, there will be $b-j+1$ options for $a$. 
Finally, when $b=d$, we must have $a\in \{j, \ldots, d-1\}$ since there is no $w_{dd}$ variable, in particular there are $d-j$ options for $a$.
That is, 
\begin{align*}
|T^\max_{ij}|&=\left(\sum_{b=1}^{j-1} b \right)+(i+1)+\left( \sum_{b=j+1}^{d-1} b-j+1\right)+(d-j)\\
&=\frac{j(j-1)}{2}+(i+1)+\frac{(d-j-1)(d-j+2))}{2}+(d-j)\\
&=\frac{d^2-2dj+3d+2j^2-4j+2i}{2}
\end{align*}
\end{proof}

\begin{corollary}\label{corij}
Adopt  \Cref{def:TandS}.
Let $1\leq i<j$ and $j\geq 4$.
If  $T_{ij}^{k_2l_2}>_\tau T_{ij}^{k_1l_1}$ are nonempty and consecutive with respect to the total ordering $\tau$, then $|T_{ij}^{k_2l_2}|=|T_{ij}^{k_1l_1}|+1$.
\end{corollary}
\begin{proof}
Let $T_{ij}^{k_2l_2}>_\tau T_{ij}^{k_1l_1}$ be nonempty and consecutive with respect to the total ordering $\tau$.
Then $w_{k_1 l_1}<_\omega w_{k_2 l_2}$ are consecutive in $S_{ij}$ with respect to $\omega$.
By the same reasoning as the proof of  \Cref{Tmax}, if $j\leq \min\{a,b\}$, then $w_{ab}w_{k_il_i}w_{ij}\in  T^{k_il_i}_{ij}$. 

We now consider the $w_{ab}<_\omega w_{ij}$ and claim that if $w_{ab}w_{k_2 l_2}w_{ij}\in T_{ij}^{k_2l_2}$ with $w_{ab}\neq w_{k_2l_2}$, there is a corresponding element $w_{ab}w_{k_1 l_1}w_{ij}\in T_{ij}^{k_1l_1}$.
Suppose, for contradiction, that $\pmb{w}=w_{ab}w_{k_1l_1}w_{ij}\not \in T_{ij}^{k_1l_1}$.
Now, $\pmb{w}$ must be in some $T_{*}^{*}$ set since $w_{k_1l_1}w_{ij}\in [\init(\Lambda)]_2$. 
Since $w_{ab}, w_{k_1,l_1}<_\omega w_{ij}$,  and $\pmb{w}\not \in T_{ij}^{k_1l_1}$, then $\pmb{w} \in T_{ij}^{ab}$.
This further implies $w_{ab}\geq_\omega w_{k_2l_2}$.
That is, $w_{ab} =w_{k_2l_2}$ or $w_{ab}w_{k_2 l_2}w_{ij}\in T_{ij}^{ab}$ with $w_{ab}>_\omega w_{k_2l_2}$.
In both cases, we have a contradiction.
In particular, whenever $w_{ab}<_\omega w_{ij}$, if $w_{ab}w_{k_2 l_2}w_{ij}\in T_{ij}^{k_2l_2}$ such that $w_{ab}\neq w_{k_2l_2}$, there is a corresponding element $w_{ab}w_{k_1 l_1}w_{ij}\in T_{ij}^{k_1l_1}$.

On the other hand, $w_{k_2 l_2}^2w_{ij}\in T_{ij}^{k_2l_2}$, but there is no corresponding $w_{k_1l_1}w_{k_2 l_2}w_{ij}\in T_{ij}^{k_1l_1}$ since this element is in the larger set $T_{ij}^{k_2l_2}$.
Therefore, $|T_{ij}^{k_2l_2}|=|T_{ij}^{k_1l_1}|+1$, as claimed.
\end{proof}

We are now prepared to calculate the size of $T$, namely the sum of all $|T_{ij}^{kl}|$ whenever $w_{ij}w_{kl}\in [\init(\Lambda)]_2$.
We have shown that for a fixed pair $i,j$, the values $|T_{ij}^{kl}|$ are consecutive integers when ordered according to $\tau$.
We have also described that consecutive sum of integers.
By construction, there are exactly $|S_{ij}|$ many sets of the form $T_{ij}^{kl}$ and we calculated the value of the largest such $|T_{ij}^{kl}|$ for a fixed pair $i,j$, namely $|T^\max_{ij}|$.

\begin{theorem}\label{consecutiveSum}
Adopt  \Cref{def:TandS}.
For a fixed $i\leq j$ with $S_{ij}\neq \emptyset$, the partial sum $\displaystyle \sum_{k,l} |T_{ij}^{kl}|$ is calculated by a sum of consecutive integers. In particular, 
\[\sum_{k,l} |T_{ij}^{kl}|=\dfrac{(|S_{ij}|)(2|T^\max_{ij}|-|S_{ij}|+1)}{2}.\]
Further, when considering the sum of all possible $T_{ij}^{kl}$ we have, 
\[|T|=\sum_{i,j,k,l} |T_{ij}^{kl}|=\frac{14d^6+30d^5-40d^4-330d^3-694d^2+1740d-4320}{6!}.\]
\end{theorem}
\begin{proof}
Recall that if $w_{k_1 l_1}<w_{k_2 l_2}$ are consecutive in $S_{ij}$ then $|T_{ij}^{k_2l_2}|=|T_{ij}^{k_1l_1}|+1$ by  \Cref{corij}.
We also know that there are exactly $|S_{ij}|$ many sets of the form $T_{ij}^{kl}$. 
The largest number in the consecutive sum is $|T^\max_{ij}|$. 
Thus smallest number in the consecutive sum is $|T^\max_{ij}|-|S_{ij}|+1$. 
Hence, for a fixed $i,j$, the partial sum is  \[\displaystyle \sum_{k,l} |T_{ij}^{kl}|=\frac{(|S_{ij}|)(|T^\max_{ij}|+(|T^\max_{ij}|-|S_{ij}|+1))}{2}=\frac{(|S_{ij}|)(2|T^\max_{ij}|-|S_{ij}|+1)}{2}.\]

The proof that the value $|T|$ is as claimed is an algebraic manipulation of finite sums and some details will be omitted.
In the calculation below, we assume without loss of generality that $i< j$ and $j\geq 4$.
Substituting in the values for $|S_{1j}|, |S_{ij}|, |T^\max_{1j}|, |T^\max_{ij}|$ in  \Cref{S1j,Sij,sizeT} we see the following.
\begin{align*}
\text{For } i=1: \hspace{.1in} \sum_{k,l} |T_{1j}^{kl}|&=\frac{(\frac{j(j-3)}{2})[2(\frac{d^2-2dj+3d+2j^2-4j+2}{2})-\frac{j(j-3)}{2}+1]}{2}\\
	&=\frac{j(j-3)[2(d^2-2dj+3d+2j^2-4j+2)-j(j-3)+2]}{8}\\
\text{For } i\neq1: \hspace{.1in} \sum_{k,l} |T_{ij}^{kl}|&=\frac{\frac{(j-i)(i+j-1)}{2}[2(\frac{d^2-2dj+3d+2j^2-4j+2i}{2})-\frac{(j-i)(i+j-1)}{2}+1]}{2}\\ 
	&=\frac{(-2)i^4}{8}+\frac{(-2d^2+4dj-6d+4j)i^2}{8}+\frac{(2d^2-4dj+6d+4j^2-8j+2)i}{8}\\
	&\hspace{.2in}+\frac{(2 d^2 j^2 - 2 d^2 j - 4 d j^3 + 10 d j^2 - 6 d j + 2 j^4 - 8 j^3 + 8 j^2 - 2 j)}{8}\end{align*}

Now we have,
\begin{align*}
|T|&=\sum_{i,j,k,l} |T_{ij}^{kl}|=\sum_{j=4}^d \left( \sum_{i=1}^{j-1} \left (\sum_{k,l} |T_{ij}^{kl}|\right)\right)	\\
&=\sum_{j=4}^d  \left(\frac{32}{120}\right)j^5+\left(\frac{-40d-105}{120}\right)j^4+\left(\frac{20d^2+240d+70}{120}\right)j^3\\
	&\hspace{.2in}+\left(\frac{-30d^2+10d+45}{120}\right)j^2+\left(\frac{-50d^2-150d-162}{120}\right)j 	\\
&=\frac{14d^6+30d^5-40d^4-330d^3-694d^2+1740d-4320}{6!}
\end{align*}

\end{proof}

We will now count those elements of $[\init(\Lambda)]_3$ which are not in $T$. 
Let $G=[\init(\Lambda)]_3\setminus T$. 
Clearly, $G$ and $T$ are disjoint sets in $[\init(\Lambda)]_3$, so in particular $HF_{\Lambda}(3)\geq |G|+|T|$.
We will compute a lower bound for $|G|$ and show that when added to $|T|$ computed above, we obtain exactly $HF_{I(X)}(3)$.
This implies $ HF_{\Lambda}(3)=HF_{I(X)}(3)$, completing the proof of the main theorem.
We will begin with defining some sets of degree three monomials.

\begin{definition}\label{Gi}
Adopt  \Cref{familydata}.
Let $G=[\init(\Lambda)]_3\setminus T$ and 
\begin{align*}
G_1&= \{w_{13}w_{23}^2, w_{13}^2w_{23}, w_{22}w_{13}w_{23}, w_{23}^3,w_{12}w_{1d}^2, w_{22}w_{1d}^2 \} \\
G_2&=\{w_{1i} w_{1j} w_{1k} \hspace{.05in} \text{such that} \hspace{.05in} 3\leq i\leq j\leq k \leq d\} \\
G_3&=\{w_{13}w_{23} w_{ij} \hspace{.05in} \text{such that} \hspace{.05in} 3\leq i \leq j \leq d\}\\
G_4&=\{w_{23}^2 w_{ij} \hspace{.05in} \text{such that} \hspace{.05in} 3\leq i \leq j \leq d\}.
\end{align*}

\end{definition}
Notice that $G_i$ are disjoint sets.
We will show that $G_i\subset G$ and so $|G|\geq \sum_{i=1}^4 |G_i|$.
For each monomial $\pmb{w} \in G_i$ we will show that $\pmb{w} \not \in T$ and
we will find a corresponding element $F\in \Lambda$ such that $\init_\omega(F)=\pmb{w}$.
The following remark will list generators used in these proofs.
We will \underline{mark} the leading term, when known.
We typically write the monomials order by $\omega$ from left to right to more easily observe the leading terms.

\begin{remark}\label{rem:GBgens}
Select the $4\times 4$ principal submatrix $[12ij;]$ of $\mathcal M$ such that $i<j$. 
Recall that the sign of the terms comes from their relative position inside the principal submatrix. 
As an abuse of notation, we include the option for $j=d$, however, to obtain a valid generator, we must first set the variable $w_{dd}=0$. 
In each of the elements below, the leading term is \underline{marked}.
\begin{enumerate}[(A)]
\item Elements of $\Lambda_0$
\begin{itemize}
\item $f_1(i,j)=[12|ij]=\underline{ w_{1i}w_{2j}}-w_{2i}w_{1j}$
\item $f_2(i,j)=[1i|2j]=-\underline{ w_{2i}w_{1j}}+w_{12}w_{ij}$
\item $f_3(i,j)=[1j|2i]=-\underline{ w_{1i}w_{2j}}+w_{12}w_{ij}$
\end{itemize}
\item Elements of $\Lambda_1$ 
\begin{itemize}
\item $g_{1}(i,j)=[2i|2j]-[1j|1i]=\underline{ w_{2i}w_{2j}}-w_{1i}w_{1j}-w_{22}w_{ij}+w_{11}w_{ij}$ 
\item $g_{2}(i,j)=[ij|2i]+[12|1j]=-\underline{ w_{ii}w_{2j}}+w_{2i}w_{ij}-w_{12}w_{1j}+w_{11}w_{2j}$ 
\item $g_{3}(i,j)=[ij|2j]-[12|1i]=\underline{w_{2j}w_{ij}}-w_{2i}w_{jj}-w_{12}w_{1i}+w_{11}w_{2i}$ 
\item $g_{4}(i,j)=[ij|1i]-[12|2j]=\underline{w_{ii}w_{1j}}-w_{1i}w_{ij}-w_{22}w_{1j}+w_{12}w_{2j}$ 
\item $g_{5}(i,j)=[12|2i]+[ij|1j]=-\underline{w_{1j}w_{ij}}+w_{1i}w_{jj}-w_{22}w_{1i}+w_{12}w_{2i}$ 
\item $g_6(i,j)=[2j|1j]-[1i|2i]\pm=\underline{ w_{1j}w_{2j}} -w_{1i}w_{2i}-w_{12}w_{jj}+w_{12}w_{ii}$ 
\end{itemize}

\item Elements of $\Lambda_2$ 
\begin{itemize}
\item $h_1(i,j)=([1j]+[2i])-([12]+[ij])=\underline{ w_{ij}^2}-w_{1j}^2-w_{ii}w_{jj}-w_{2i}^2+w_{22}w_{ii}+w_{12}^2+w_{11}w_{jj}-w_{11}w_{22}$ 
\item $h_2(i,j)=([1j]+[2i])-([1i]+[2j])=\underline{ w_{2j}^2}-w_{1j}^2-w_{2i}^2+w_{1i}^2-w_{22}w_{jj}+w_{22}w_{ii}+w_{11}w_{jj}-w_{11}w_{ii}$ 
\end{itemize}

\end{enumerate}
\end{remark}

\begin{lemma}\label{g1}
Adopt  \Cref{Gi}. There is a containment of sets $G_1\subseteq G$.
In particular, $|G_1|=6$. 
\end{lemma}

\begin{proof}
We begin by showing that if $\pmb{w}\in G_1$, then $\pmb{w}\not\in T$.
Since $S_{13}=S_{23}=S_{22}=\emptyset$, the monomials $w_{13}w_{23}^2, w_{13}^2w_{23}, w_{22}w_{13}w_{23}, w_{23}^3\not \in T$.  
Similarly, since $S_{12}=S_{22}=\emptyset$ and $w_{12},w_{22}, w_{1d}\not\in S_{1d}$, the monomials $w_{12}w_{1d}^2, w_{22}w_{1d}^2 \not \in T$.

We will now show that if $\pmb{w}\in G_1$, then it is the leading monomial of an element $F_i\in \Lambda$ listed below.
Notice that these are obtained from elements in  \Cref{rem:GBgens}.
In particular, after setting $w_{dd}=0$, 
\begin{align*}
F_1=&-w_{2d} \cdot f_2(3,d)-w_{23}\cdot g_6(3,d)=\underline{w_{13}w_{23}^2}-w_{2d}w_{12}w_{3d}  -w_{23}w_{12}w_{33}\\
F_2=& -w_{1d} \cdot f_3(3,d)-w_{13}\cdot g_6(3,d))= \underline{w_{13}^2w_{23}}-w_{1d} w_{12}w_{3d}-w_{13}w_{12}w_{33}  \\                   
F_3=&  w_{3d} \cdot f_2(3,d)-w_{23}\cdot g_5(3,d)= \underline{w_{23}w_{22}w_{13} }+w_{12}w_{3d}^2  -w_{23}w_{12}w_{24}  \\
F_4=& w_{1d}\cdot f_2(3,d)+  w_{2d}\cdot g_1(3,d)-w_{23}\cdot h_2(3,d)\\
=& \underline{w_{23}^3}-w_{13}w_{1d}w_{2d}-w_{13}^2w_{23}-w_{22}w_{2d}w_{3d}-w_{22}w_{23}w_{33}+w_{12}w_{1d}w_{3d}+w_{11}w_{2d}w_{3d}+w_{11}w_{23}w_{33}\\
F_5=& w_{3d}\cdot f_3(3,d)+w_{13}\cdot g_3(3,d)-w_{12}\cdot h_1(3,d)\\
=&\underline{w_{12}w_{1d}^2}+w_{12}w_{23}^2-w_{12}w_{13}^2-w_{12}w_{22}w_{33}-w_{12}^3+w_{11}w_{13}w_{23}+w_{11}w_{12}w_{22}\\
\end{align*}
\begin{align*}
F_6=&-w_{3d}\cdot g_1(3,d)+ w_{23}\cdot g_3(3,d)+ w_{13}\cdot g_5(3,d) -w_{22}h_1(3,d)\\
   =&\underline{w_{22}w_{1d}^2}+w_{22}w_{23}^2-w_{22}w_{13}^2-w_{22}^2w_{33}-w_{22}w_{12}^2-w_{11}w_{3d} ^2    +w_{11}w_{23}^2   +w_{11}w_{22}^2     \\
\end{align*}
Since $F_i\in \Lambda$ and  $G_1=\{\init(F_i)\}$, we see that $G_1\subset G$, as claimed.
The size $|G_1|=6$ is clear.
\end{proof}

\begin{lemma}\label{g2}
Adopt  \Cref{Gi}. There is an inclusion of sets $G_2 \subset G$.
Further \[|G_2|=(d-2)+(d-2)(d-3)+{{d-2}\choose 3}.\]
\end{lemma}

\begin{proof}
Let $3\leq i\leq j\leq k \leq d$.
We begin by showing that if $\pmb{w}\in G_2$, then $\pmb{w}\not\in T$.
By  \Cref{S1j}, we see $w_{1i},w_{1j}\not \in S_{1k}$. 
In particular, $w_{1i} w_{1j} w_{1k}\not \in T$.

We will now show that if $\pmb{w}\in G_2$, then it is the leading monomial of an element $F_i\in \Lambda$ listed below.
Notice that these are obtained from elements in  \Cref{rem:GBgens} and will depend on potential equalities of $i,j,k$.
For $F_1$, let $3\leq i\leq j<k\leq d$; for $F_2$, let $3\leq i<j=k\leq d$; for $F_3$, let $3< i=j=k\leq d $; and for $F_4$, let $i=j=k=3$.
Notice that if $j=d$ the proof is the same except for setting $w_{jj}=w_{dd}=0$, which will not affect any leading terms.

\begin{align*}
F_1&=w_{2j}\cdot f_3(i,k)-w_{1i}\cdot g_1(j,k))=\underline{w_{1i}w_{1j}w_{1k}}+w_{22}w_{1i}w_{jk}+w_{12}w_{2j}w_{ik}-w_{11}w_{1i}w_{jk}\\
F_2&=-w_{2j}\cdot f_3(i,j)-w_{1i}\cdot h_2(i,j)=\underline{w_{1i}w_{1j}^2}+w_{1i}w_{2i}^2\\
&-w_{1i}^3+w_{22}w_{1i}w_{jj}-w_{22}w_{1i}w_{ii}-w_{12}w_{2j}w_{ij}-w_{11}w_{1i}w_{jj}+w_{11}w_{1i}w_{ii}\\
F_3&=w_{2j}\cdot g_6(3,j)-w_{1j}\cdot h_2(3,j)=\underline{w_{1j}^3}+w_{23}^2w_{1j}-w_{13}w_{23}w_{2j}-w_{13}^2w_{1j}\\
&+w_{22}w_{1j}w_{jj}-w_{22}w_{33}w_{1j}-w_{12}w_{2j}w_{jj}+w_{12}w_{33} w_{2j} -w_{11}w_{1j}w_{jj}+w_{11}w_{33}w_{1j}\\
F_4 &=-w_{2d}\cdot f_1(3,d)-w_{2d}\cdot f_2(3,d)-w_{1d}\cdot g_1(3,d)-w_{23}\cdot g_6(3,d)+w_{13}\cdot h_2(3,d)\\
&=\underline{w_{13}^3}+w_{22}w_{1d}w_{3d}+w_{22}w_{13}w_{33}-w_{12}w_{2d}w_{3d}-w_{12}w_{23}w_{33}-w_{11}w_{1d}w_{3d}-w_{11}w_{13}w_{33}
\end{align*}
Since $F_i\in \Lambda$ and  $G_2=\{\init(F_i)\}$, we see that $G_2\subset G$, as claimed.

Regarding the size of $G_2$ notice that either $i=j=k$ , two are equal, or they are all distinct. 
If $i=j=k$, there are $d-2$ elements since $i=j=k$ can range from $3$ to $d$.
If two are equal, there are $d-2$ ways to select the index for the squared term and $d-3$ ways to select the other index.
This gives a total of $(d-2)(d-3)$ elements.
Finally, if all indices are distinct, there are  ${d-2\choose 3}$ ways to select them.
The sum gives the size of $G_2$ as claimed.
\end{proof}

 \begin{lemma}\label{g3}
 Adopt  \Cref{Gi}. There is an inclusion of sets $G_3 \subset G$.
 Further $\displaystyle|G_3|=\frac{d(d-3)}{2}$.
 \end{lemma}
 
 \begin{proof}
 Let $3\leq i \leq j \leq d$.
 We begin by showing that if $\pmb{w}\in G_3$, then $\pmb{w}\not\in T$.
By  \Cref{S_empty}, we see $S_{13},S_{23}, S_{ii}$ are all empty.
In particular, if  $ w_{13}w_{23} w_{ij}\in T$ then $3<i<j\leq d$ with $w_{13}$ or $w_{23}$ in $S_{ij}$.
However, by \Cref{Sij}, this would require $i<\max\{1,2,3\}$ or $\max\{1,2,3\}=j$, which is not possible. 
In particular, $w_{13}w_{23} w_{ij}\not \in T$.

We will now show that if $\pmb{w}\in G_3$, then it is the leading monomial of an element $F_i\in \Lambda$ listed below.
Notice that these are obtained from elements in  \Cref{rem:GBgens} and will depend on potential equalities  of $i,j$.
For $F_1$, let $3<i<j$; for $F_2$, let $3=i<j$; and for $F_3$, let $3\leq i=j <d$.
Notice that if $j=d$ the proof is the same except for setting $w_{jj}=w_{dd}=0$, which will not affect any leading terms.

\begin{align*}
F_1&=w_{13}\cdot [2j|3i]-w_{3j}\cdot f_3(3,i)=\underline{w_{13}w_{23}w_{ij}}-w_{12}w_{3i}w_{3j}\\
F_2 &=-w_{33}\cdot f_3(3,j)+w_{13}\cdot g_2(3,j)\\
&=\underline{w_{13}w_{23}w_{3j}}-w_{12}w_{33}w_{3j}-w_{12}w_{13}w_{1j}+w_{11}w_{13}w_{2j}\\
F_3&=w_{1i}\cdot g_3(i,d)+w_{2d}\cdot g_4(i,d)-w_{ii}\cdot g_6(3,d)\\
&=\underline{w_{13}w_{23}w_{ii}}-w_{22}w_{1d}w_{2d}+w_{12}w_{2d}^2-w_{12}w_{ii}^2  -w_{12}w_{1i}^2+w_{11}w_{1i}w_{2i}  \\\end{align*}
Since $F_i\in \Lambda$ and  $G_3=\{\init(F_i)\}$, we see that $G_3\subset G$, as claimed.

 Regarding the size of $G_3$, notice that when $i=j$, there are $d-3$ possibilities since both cannot be $d$.
When $i< j$, for each fixed $i$, there are $d-i$ possibilities for $j$.
Therefore $|G_3|=(d-3)+\sum_{i=3}^{d-1}(d-i)=(d-3)+\frac{(d-2)(d-3)}{2}=\frac{d(d-3)}{2}$, as claimed.
 \end{proof}

 \begin{lemma}\label{g4}
 Adopt  \Cref{Gi}. 
There is an inclusion of sets $G_4\subset G$.
Further  $|G_4|=\frac{d(d-3)}{2}$.
 \end{lemma}
 
 \begin{proof}
 We begin by showing that if $\pmb{w}\in G_4$, then $\pmb{w}\not\in T$.
Let $3\leq i \leq j \leq d$.
By  \Cref{S_empty}, we see $S_{23}, S_{ii}$ are empty.
In particular, if  $w_{23}^2 w_{ij}\in T$ then $3<i<j\leq d$ with  $w_{23} \in S_{ij}$.
However, by \Cref{Sij}, this would require $i<\max\{2,3\}$ or $\max\{2,3\}=j$, which is not possible. 
In particular, $w_{23}^2 w_{ij}\not \in T$.

We will now show that if $\pmb{w}\in G_4$, then it is the leading monomial of an element $F_i\in \Lambda$ listed below.
Notice that these are obtained from elements in  \Cref{rem:GBgens} and will depend on the relative sizes of $i,j$.
For $F_1$, let $3<i<j$; for $F_2$, let $3=i<j$; and for $F_3$.
Notice that if $j=d$ the proof is the same except for setting $w_{jj}=0$.

\begin{align*}
F_1&= w_{23}\cdot [2i|3j]+ w_{3i}\cdot g_1(3,j)\\
&=\underline{w_{23}^2w_{ij}}-w_{13}w_{3i}w_{1j}-w_{22}w_{3i}w_{3j}+w_{11}w_{3i}w_{3j}\\
F_2&=w_{33}\cdot g_1(3,j)+w_{23}\cdot g_2(3,j)\\
&=\underline{ w_{23}^2w_{3j}}-w_{13}w_{33}w_{1j}-w_{22}w_{33}w_{3j}-w_{12} w_{23}w_{1j}+w_{11}w_{33}w_{3j}+w_{11} w_{23}w_{2j}\\
F_3&=- w_{2d}\cdot g_2(i,d)+w_{2i} \cdot g_{3}(i,d)-w_{1d}\cdot g_{4}(i,d)+w_{1i}\cdot g_{5}(i,d)-w_{ii}\cdot h_2(3,d)\\
&=\underline{w_{23}^2w_{ii}}-w_{13}^2w_{ii}-w_{22}w_{33}w_{ii}+w_{22}w_{1d}^2-w_{22}w_{1i}^2-w_{11}w_{2d}^2+w_{11}w_{2i}^2+w_{11}w_{33}w_{ii}\\
\end{align*}

 Regarding the size of $G_4$, notice that when $i=j$, there are $d-3$ possibilities since both cannot be $d$.
When $i< j$, for each fixed $i$, there are $d-i$ possibilities for $j$.
Therefore $|G_4|=(d-3)+\sum_{i=3}^{d-1}(d-i)=(d-3)+\frac{(d-2)(d-3)}{2}=\frac{d(d-3)}{2}$. \end{proof}

 \begin{lemma}\label{thm:g}
 Adopt  \Cref{Gi}.
 Then 
 \[\sum_{i=1}^4 |G_i|= \frac{120d^3+360d^2-1920d+4320}{6!}\]
 \end{lemma}
 
 \begin{proof}
By  \Cref{g1,g2,g3,g4},
 \begin{align*}
 \sum_{i=1}^4 |G_i|&=6+\left((d-2)+(d-2)(d-3)+{{d-2}\choose 3}\right)+\frac{d(d-3)}{2}+ \frac{d(d-3)}{2}\\
 &=\frac{120d^3+360d^2-1920d+4320}{6!}
 \end{align*}

 \end{proof}
 
\begin{theorem}\label{HF3}
The Hilbert Functions of $\Lambda$ and $I(X)$ are equivalent in degree 3, that is $HF_{\Lambda}(3)=HF_{I(X)}(3)$.
\end{theorem}
\begin{proof}
Recall that $\Lambda\subset I(X)$, so we have $HF_{\Lambda}( 3)\leq HF_{I(X)}(3)$.
Now, based on the definitions of $G$ and $T$, we have $HF_{\Lambda}(3)\geq|T|+|G|\geq |T|+ \sum_{i=1}^4 |G_i|$.
These values were computed in  \Cref{consecutiveSum,thm:g}, respectively.
Indeed,
\begin{align*}
|T|+ \sum_{i=1}^4 |G_i|&=\frac{14d^6+30d^5-40d^4-330d^3-694d^2+1740d-4320}{6!} \\
& \hspace{.3in}+\frac{120d^3+360d^2-1920d+4320}{6!}\\
&=\frac{14 d^6 + 30 d^5 - 40 d^4 - 210 d^3 - 334 d^2 - 180 d}{6!}\\
&=HF_{I(X)}(3).
\end{align*}
Therefore the Hilbert Functions of $\Lambda$ and $I(X)$ are equal in degree 3.
\end{proof}

\section*{Acknowledgments}
I would like to thank my PhD advisor Claudia Polini for suggesting this problem, her support and encouragement, and her careful reading. 
I would also like to thank Kuei-Nuan Lin for her careful reading of the manuscript and valuable suggestions. 
This project was partially supported by the Mathematical Endeavors Revitalization Program of the Association for Women in Mathematics.

\end{document}